\documentclass[11pt,twoside]{amsart}
\usepackage{eurosym}
\usepackage{amscd,amsfonts,amsmath,amssymb,amsthm,latexsym}
\usepackage{enumerate,array,srcltx,amsmath}
\usepackage{leqno}
\usepackage[dvips]{graphicx}
\usepackage{multicol}
\usepackage{a4}
\usepackage{fancyhdr}
\usepackage{caption}
\usepackage{subcaption}
\usepackage{mathrsfs}
\usepackage[driverfallback=pdfmx,,colorlinks,bookmarksopen,bookmarksnumbered,citecolor=blue, urlcolor=blue]{hyperref}
\usepackage{color}

\definecolor{darkblue}{rgb}{0,0,0.545098}
\definecolor{darkgreen}{rgb}{0,0.392157,0}
\newtheorem{theorem}{Theorem}[section]
\newtheorem{lemma}[theorem]{Lemma}
\newtheorem{proposition}[theorem]{Proposition}
\theoremstyle{definition}

\theoremstyle{remark}
\newtheorem{remark}[theorem]{Remark}
\numberwithin{equation}{section}

\makeatother
\input{tcilatex}

\begin{document}
\title[Upper Bounds for a nonlinearly damped]%
{Nonlinear damping effects for the 2D Mindlin-Timoshenko system}
\author[Bchatnia ]{ Ahmed Bchatnia$^1$}
\address{$^1$ LR Analyse non-lin\'eaire et g\'eom\'etrie, LR21ES08, 
              Department of Mathematics, Faculty of Sciences of Tunis,
              University of Tunis El-Manar, 2092 El Manar II, Tunisia}
\email{ahmed.bchatnia@fst.utm.tn}
\author[Chebbi]{Sabrine chebbi$^1$}
\address{$^1$  LR Analyse non-lin\'eaire et g\'eom\'etrie, LR21ES08, 
              Department of Mathematics, Faculty of Sciences of Tunis,
              University of Tunis El-Manar, 2092 El Manar II, Tunisia}
\email{sabrine.chebbi@fst.utm.tn}
\author[Hamouda]{Makram Hamouda$^{ 2^{\ast}}$}
\address{$^2$ Department of Basic Sciences, Deanship of Preparatory Year and Supporting Studies, Imam Abdulrahman Bin Faisal University, P. O. Box 1982, Dammam, Saudi Arabia}
\email{mmhamouda@iau.edu.sa}
\address{$^*$ Corresponding author.}

\maketitle

\begin{abstract}
We study in this article the asymptotic behavior of the Mindlin-Timoshenko system subject to a nonlinear dissipation acting only on the equations of  the rotation angles. First, we briefly recall the existence of the solution of this system. Then, we prove that the energy associated with the Mindlin-Timoshenko system  fulfills a dissipation relationship showing that the energy is decreasing. Moreover, when the wave speeds are equal, we establish an explicit and general decay result for the energy.
\end{abstract}

\noindent{\emph{Subjectclass[2020]: }}{35B40, 93D15, 93D20.}

\section{\textbf{Introduction}}
 In the recent years, the Timoshenko system has been attracting a considerable attention. Indeed, this system is important not only from the mathematical point of view but also from the physical perspective in view of the multiple applications in mechanics and aeronautics among many other sciences. The Timoshenko system is considered also as  one of the simplest  models   which  describes the transverse vibration of a beam with fixed extremities.
The transverse vibration  of a beam under the Timoshenko approach  is mathematically described by the following coupled system:

\begin{equation*}\left\{\begin{array}{ll}\rho \varphi _{tt}-K(\varphi _{x}+\psi )_{x}=0, &\text{in} \ \left( 0,L\right) \times
\left( 0,\infty \right), \\
I_{\rho } \psi _{tt}-(EI\psi _{x})_{x}+K (\varphi _{x}+\psi) =0,&\text{in} \  \left(0,L\right) \times \left( 0,\infty \right).%
\end{array}%
\right.  \label{Timo11}
\end{equation*}
 The  functions $\varphi$ and $\psi$ denote, respectively, the transverse displacement of the beam and the rotation angle of  the filament. 
 
 In  \cite{3,6}, the authors consider a one-dimensional system with a linear damping term given as follows:
 \begin{equation} \label{2}
     \left \{ \begin{array}{lrl}
 \rho_1 \varphi_{tt} - k(\varphi_x+ \psi)_x=0,  \hspace{2.6cm } (x,t) \in \left(0,L\right)\times \mathbb{R}_+,\\
         \rho_2\psi_{tt}-b\psi_{xx}+k(\varphi_x+\psi)+d \psi_t=0, \hspace{0.5cm }(x,t) \in \left(0,L\right)\times \mathbb{R}_+.
       \end{array}\right.
       \end{equation}
 They proved that the solution of the system \eqref{2} is exponentially stable if and only if the wave speeds are equal $(\frac{k}{\rho_1}=\frac{b}{\rho_2}).$ Such a result has been the common point in several works \cite{1,2,3,4,5} with  different types of dissipation. Moreover, for the damped Timoshenko system (namely \eqref{2} with $d \psi_t$ being replaced by $g(\psi_t)$) with  a nonlinear dissipation $g(\psi_t)$ and without any growth assumption at the origin on the function $g$, Alabau-Boussouira  established in \cite{AB2}  a general semi-explicit formula for the decay rate of the energy at infinity in the case of  equal wave-speeds, and she  proved a polynomial decay in the case of different speeds of propagation, for both linear and nonlinear globally Lipschitz feedback $g(\psi_t)$. Then, in  \cite{8}, the same author established a strong lower energy estimate for  
the solutions of the nonlinearly damped Timoshenko beams after proving some additional regularity results.\\ 

Recently, the authors in  \cite{me} proved a general semi-explicit lower estimate for the energy associated with  the  nonlinear Timoshenko system with thermoelasticity coupled with the heat flux which is given by  Cattaneo's law. This result together with the one  on the upper estimate for the energy of the above mentioned system, obtained in \cite{ayadi}, constitute an interesting step  to better understanding the optimality in terms of {\it optimal feedback}.   \\

Based on these simple observations in the history of the problem, some questions naturally arise, we list them in the following.
\begin{itemize}
\item[$\mathcal{Q}1-$] Can we generalize the results obtained for the Timoshenko beam system in dimension one to higher dimensions (mainly the dimension two (2D) which will be considered in the present article) by investigating the upper and lower estimates, and then target the optimality?
\item[$\mathcal{Q}2-$] What should be the "sufficient" dissipation which allows the stabilization of the   2D system and that generates the exponential decay of the energy? Does the stability depend on any relationship between the physical coefficients in the system?
\item[$\mathcal{Q}3-$] In the case of a non-exponential decay rate, what is the type of decay (polynomial or others) and how can we  investigate the optimality?
\end{itemize}

Of course  to answer all these questions, we need to start by the most necessary ones.  Hence, we consider in this article, as  an extension of the one-dimensional Timoshenko system, the 2D	model which  is  subsequently developed by Reissner \cite{46} and Mindlin \cite{37}. The Reissner-Mindlin-Timoshenko plate theory can be viewed as a two-dimensional generalization of the thin beam model. Indeed, this model takes into account the shear deformations in addition to the assumption  that the filaments of the plate must remain perpendicular to its mid-plane.\\

Mathematically, the Reissner-Mindlin-Timoshenko model is actually a hyperbolic system composed of three coupled second-order partial differential equations in space dimension two:
  \begin{equation}\label{IMT1}
 \rho_1w_{tt}-K(\psi+ w_x)_x+ (\varphi+w_y)_y=0,
 \end{equation}
 \begin{equation}\label{IMT2}
 \rho_2 \psi_{tt}-D \psi_{xx}- D\left(\frac{1-\mu}{2}\right) \psi_{yy}- D\left(\frac{1+\mu}{2}\right)  \varphi_{xy}+K(\psi+ w_x)=0,
 \end{equation}
 \begin{equation}\label{IMT3}
 \rho_2 \varphi_{tt}-D \varphi_{yy}-D\left(\frac{1-\mu}{2}\right)  \varphi_{xx}-D \left( \frac{1+\mu}{2}\right)  \psi_{xy}+K (\varphi+ w_y)=0.
 \end{equation}
Let us describe some characteristics related to the above equations which are considered on the bounded domain $\Omega \subset \mathbb{R}^2$. The functions $w$, $\psi$ and $\varphi$ depend on $(t,x,y)\in [0,\infty)\times \Omega $ and denote, receptively, the transverse displacement and the rotational angles of  the plate. The parameter $\rho$ is the (constant) mass per unit of surface area, $\mu$ is the Poisson's ratio $(0<\mu<1/2)$, $D=\frac{Eh^3}{12(1-\mu^2)}$ is the modulus of flexural rigidity where $h$ is the (uniform) plate thickness and $K=\frac{kEh}{2(1+\mu)}$ is the shear modulus where $E$ is the Young's modulus and $k$ is the shear correction. The  speeds of wave propagation for the Reissner-Mindlin-Timoshenko system are given as follows:
\begin{equation*}
v_1^2:=\frac{K}{\rho_1}\quad  \ \text{and} \quad v_2^2:=\frac{D}{\rho_2}.
\end{equation*}
Concerning the stability results of the system \eqref{IMT1}--\eqref{IMT3}, the most known ones are due to Lagnese \cite{lag2} where a bounded domain $\Omega$ with a Lipschitz boundary $\Gamma$  is considered. Moreover, the boundary $\Gamma$  is taken such that $\Gamma=\overline{\Gamma_0} \cup \overline{\Gamma_1}$ where $\Gamma_0$ and $\Gamma_1$ are relatively open, disjoint subsets of $\Gamma$ with $\Gamma_1 \neq \varnothing$. In addition to the aforementioned assumptions on the domain $\Omega$, Lagnese considered the equations \eqref{IMT1}--\eqref{IMT3} with the following boundary conditions:
 \begin{equation*}
w=\psi=\varphi=0, \ \ \text{on} \  \Gamma_0,
\end{equation*}
\begin{equation*}
K\left( \frac{\partial w}{\partial x}+\psi, \frac{\partial w}{\partial y} + \varphi \right)\cdot v=m_1,  \ \ \text{on} \  \Gamma_1,
\end{equation*}
\begin{equation*}
D\left( \frac{\partial \psi}{\partial x}+\mu \frac{\partial\varphi}{\partial y}, \frac{1-\mu}{2}\left( \frac{\partial \varphi}{\partial x}+ \frac{\partial \psi}{\partial y}\right)\right)\cdot v= m_2,  \ \ \text{on} \  \Gamma_1,
\end{equation*}
\begin{equation*}
D\left( \frac{1-\mu}{2}\left( \frac{\partial\varphi}{\partial x}+\frac{\partial \psi}{\partial y}, \frac{\partial \varphi}{\partial y}+\mu \frac{\varphi}{\partial y}+\mu \frac{\partial \psi}{\partial x}\right)\right)\cdot v=m_3,  \ \ \text{on} \  \Gamma_1,
\end{equation*}
where $v=(v_1,v_2)$ is the unit exterior normal to $\Gamma$ and $m_1, m_2$ and $m_3$ denote the linear boundary dissipation given by $$\lbrace m_1,m_2,m_3\rbrace=- F \lbrace w_t,\psi_t,\varphi_t\rbrace,$$
with $F=[f_{ij}]$ is a $3\times 3$ matrix of real $L^{\infty}(\Gamma_1)-$functions such that $F$ is symmetric and positive semi-definite on $\Gamma_1$. Lagnese \cite{lag2} proved that the undamped system is \textit{exponentially stable}, without any restriction on the coefficients of the system. Later, Rivera and  Oquendo \cite{10} consider the system \eqref{IMT1}--\eqref{IMT3} with a boundary dissipation of memory type. They proved the exponential stability when the kernels own an exponential behavior and a polynomial one when the kernels are of polynomial type.
Let us further mention some other interesting  results as e.g.  \cite{8}, the  Reissner-Mindlin-Timoshenko system \eqref{IMT1}--\eqref{IMT3}, with frictional dissipation acting on the equations of the rotation angles, is not exponentially stable independently of any relationship between the constants of the system.  This is not the case  in \cite{12} where  the existence of a critical number, that exponentially stabilizes the system \eqref{IMT1}--\eqref{IMT3} with  the frictional dissipation acting only on the equation  of the transverse displacement, is proven. 
In fact, in \cite{12}, the Reissner-Mindlin-Timoshenko system  is damped by two linear feedback laws acting on the rotation angles and the existence of a critical number that stabilizes  the system exponentially is shown. Finally, it is worth mentioning, in view of the above cited papers among others, that the linear dissipation has been widely studied in the literature. However, there has been less focus on the interaction of nonlinear damping terms within the  2D system (the Reissner-Mindlin-Timoshenko system).

 For recent results on coupled PDE dynamics with interface on  the Reissner-Mindlin-Timoshenko plate
equations, we refer the reader to   the works by Grobbelaar-Van Dalsen \cite{MG1,MG2,MG3,MG4,MG5} and the  references therein, Giorgi and Naso \cite{Gio}, and Avalos and Toundykov \cite{Geo1,Geo}.

It is worth mentioning that the above mentioned results precise some of the necessary conditions that lead to the exponential  stabilization of the model (depending on the choice of the boundary conditions, the equality or non-equality of the wave speeds and the nature of damping terms). Here, the challenge consists in considering  nonlinear damping which are distributed everywhere in the domain and acting only on the rotational angles. Taking into account the work of  Alabau-Boussouira \cite{AB2,8}, we aim here to establish a general  and explicit decay result for the energy associated with the  system  (\ref{MTD1})--(\ref{MTD3}) below. More precisely, we prove that the  energy decay rate, as introduced in  \cite{AB2} for a nonlinearly damped hyperbolic system coupled by velocities,  can be extended to our plate model. The proofs of our results are based on multiplier techniques, weighted nonlinear integral inequalities and the optimal-weight convexity method used in \cite{AB1,AB2}. Indeed, the latter method  is originally  developed  in \cite{AB1} where the author  completed the study carried out in \cite{115}  and improved the results   in \cite{121}. In the present work we assume a convexity assumption on the feedback as we will see later on, and we prove  the   asymptotic behavior of the energy   in  higher dimensions. For more details, we refer  the reader to \cite{onestep} for  the wave equation and to \cite{AB2} for the one-dimensional Timoshenko system.

The  nonlinearly damped Reissner-Mindlin-Timoshenko system in two dimensions reads as follows:
 \begin{equation}\label{MTD1}
 \rho_1w_{tt}-K(\psi+ w_x)_x - K (\varphi+w_y)_y=0,
 \end{equation}
 \begin{equation}\label{MTD2}
 \rho_2 \psi_{tt}-D \psi_{xx}- D\left(\frac{1-\mu}{2}\right) \psi_{yy}- D\left(\frac{1+\mu}{2}\right)  \varphi_{xy}+K(\psi+ w_x)+ \chi_1(\psi_t)=0,
 \end{equation}
 \begin{equation}\label{MTD3}
 \rho_2 \varphi_{tt}-D \varphi_{yy}-D\left(\frac{1-\mu}{2}\right)  \varphi_{xx}-D \left( \frac{1+\mu}{2}\right)  \psi_{xy}+K (\varphi+ w_y)+\chi_2(\varphi_t)=0,
 \end{equation}
 where $\chi_1$ and $\chi_2$ denote the two nonlinear damping terms.
 
To the best of our knowledge, it is the first time that the asymptotic behavior of the system \eqref{MTD1}--\eqref{MTD3} (with nonlinear damping terms acting only on the rotation angles and without any damping on the displacement) is studied. \\

The paper is organized as follows. In Section \ref{s2}, we use the
nonlinear semigroup theory to obtain the well-posedness of the system \eqref{MTD1}--\eqref{MTD3}. In Section \ref{s4}, we prove that the energy of
the system \eqref{MTD1}--\eqref{MTD3} is a decreasing function of $t$. Finally, in Section \ref{examples},  we show that the optimal-weight convexity method allows us to obtain an explicit decay rate formula of  the total energy  associated with the solution of the problem \eqref{MTD1}--\eqref{MTD3}. 

\section{\textbf{Well-posedness}}\label{s2}
 We set the following assumptions on  the damping functions $\chi_i$:
\begin{equation*}
\hspace{0.3cm}(H_{0})\left\{
\begin{array}{l}
\displaystyle\chi_i \in \mathcal{C}(\mathbb{R}) \mbox{ is a monotone   increasing function}.\\
 \displaystyle\exists \ c_1, c_2 >0 \ \mbox{and a  strictly  increasing  odd function }  g \in \mathcal{C}^1(\mathbb{R})\  \mbox{such  that},\\
 \displaystyle c_1g(|s|)\leq |\chi_i(s)|\leq  c_2 g^{-1}(|s|),\hspace{2.7em}\mbox{for all} %
\hspace{1.2em}|s|\leq 1 \ \mbox{for} \  i=1,2 , \\
 \displaystyle c_{1}|s|\leq |\chi_i(s)|\leq c_{2}|s|,\hspace{6em}\mbox{ for all}\hspace{1.13em}%
|s|\geq 1  \ \mbox{ for} \ i=1,2,\\
\end{array}%
\right.
\end{equation*}%
where $g^{-1}$ denotes the inverse function of $g$.\\
Thanks to the assumption $(H_0)$, for all $\varepsilon >0$, there exist constants $c_3>0$ and $c_4>0$ such that
\begin{equation}\label{c3c4}
c_3g(|\nu|)\leq |\chi_1(\nu)|\leq c_4g^{-1}(|\nu |), \quad \forall \ |\nu |\leq \varepsilon,
\end{equation}
and two other positive constants $c_5$ and $c_6$ such that
\begin{equation}\label{c5c6}
c_6g(|\nu|)\leq |\chi_2(\nu)|\leq c_5g^{-1}(|\nu |), \quad \forall \ |\nu |\leq \varepsilon.
\end{equation}

The domain  $\Omega \subset \mathbb{R}^2$ is simply here a rectangle  given by
$$\Omega=(0,L_1)\times(0,L_2), \hspace{1cm}\mbox{with}\ L_1, L_2>0.$$
We associate with  the system \eqref{MTD1}--\eqref{MTD3} the following boundary conditions on $\Gamma = \partial\Omega$,
\begin{equation}\label{BD}
 w=\psi=\varphi=0, \quad \text{on} \ \Gamma \times \mathbb{R}^{+},
 \end{equation}
and the initial data given by
\begin{equation}\label{IC-new}
  \left\{
\begin{array}{l}
w(0,x,y)=w_0(x,y),\hspace{0.5cm} w_t(0,x,y)=w_1(x,y), \hspace{0.5cm} \mbox{in}\  \hspace{0.2cm}\Omega,\\
\psi(0,x,y)=\psi_0(x,y),\hspace{0.5cm} \psi_t(0,x,y)=w_1(x,y), \hspace{0.5cm} \mbox{in}\  \hspace{0.2cm}\Omega,\\
\varphi(0,x,y)=\varphi_0(x,y),\hspace{0.5cm} \varphi_t(0,x,y)=\varphi_1(x,y), \hspace{0.5cm} \mbox{in}\  \hspace{0.2cm}\Omega.
   \end{array}
 \right.
 \end{equation}

 In the sequel we will use the semi-group theory to show that the problem (\ref{MTD1})--(\ref{MTD3}), associated with the initial conditions \eqref{IC-new} and the Dirichlet boundary conditions \eqref{BD},  is well-posed. To this end, we write the system (\ref{MTD1})--(\ref{MTD3}) in an abstract form, using the notation $U:=(w,w_t,\psi,\psi_t,\varphi,\varphi_t)^T$, which leads to the the following Cauchy problem:
\begin{equation}\label{sst}
\left\{
\begin{array}{l}
\displaystyle \frac{dU}{dt}=\mathcal{A}U+ \mathcal{B}U, \hspace{.5cm} \mbox{for}\hspace{0.2cm} t>0,\vspace{.2cm}\\
U(0)=U_0,
\end{array}%
\right.
\end{equation}%
considered in the Hilbert  space $$\mathcal{H}:=H_0^1(\Omega)\times L^2(\Omega)\times H^1_{0}(\Omega)\times L^2(\Omega)\times H^1_{0}(\Omega)\times L^2(\Omega),$$
with $\mathcal{A}$ is the differential operator
$$ \mathcal{A}=
\left(
\begin{array}{cccccc}
0& \mathcal{I} & 0 & 0 & 0 & 0  \\
\frac{K}{\rho_1} \Delta & 0 & \frac{K}{\rho_1}\partial_x & 0 & \frac{K}{\rho_1}\partial_y& 0 \\
0 & 0 & 0 & \mathcal{I} & 0 & 0 \\
-\frac{k}{\rho_2} \partial _x & 0 & \mathcal{A}_1 & 0 & \frac{D}{\rho_2}\left(\frac{1+\mu}{2}\right)\partial_x\partial_y & 0 \\
0 & 0 & 0 & 0 & 0 & \mathcal{I} \\
-\frac{k}{\rho_2}\partial _{y} & 0  & \frac{D}{\rho_2}\left( \frac{1+\mu}{2} \right)\partial_{x}\partial_{y} & 0 &  \mathcal{A}_2 & 0 \\
\end{array}%
\right),
$$
 where $\mathcal{A}_{i}\  (i=1,2)$ are the differential operators defined by
\begin{equation*}
\mathcal{A}_1=\frac{D}{\rho_2} \left[  \partial^2_x+\left( \frac{1-\mu}{2}\right)\partial^2_y\right]-\frac{K}{\rho_2}\mathcal{I},
\end{equation*}
\begin{equation*}
\mathcal{A}_2=\frac{D}{\rho_2}\left[ \left( \frac{1-\mu}{2} \right) \partial^2_x+\partial^2_y\right]- \frac{K}{\rho_2} \mathcal{I}.
\end{equation*}
In the above $\mathcal{I}$ denotes the identity operator, and $$\mathcal{D}(\mathcal{A})=\left((H^2(\Omega)\cap H^1_0(\Omega))\times H^1_0(\Omega)\right)^3,$$
stands for the domain of $\mathcal{A}$ $:\mathcal{D}(\mathcal{A})\subset \mathcal{H}	\rightarrow \mathcal{H}$.\\
The operator $\mathcal{B}$ includes the nonlinear damping terms, and it is given by
\begin{equation*}
\displaystyle \mathcal{B}(U):=\mathcal{B}
\left(
\begin{array}{c}
u_1 \\
u_2 \\
u_3\\
u_4\\
u_5\\
u_6%
\end{array}%
\right)
=\left(
\begin{array}{c}
0 \\
0 \\
0 \\
\chi_1(u_4)\\
0 \\
\chi_2(u_6)%
\end{array}%
\right).
\end{equation*}%
The domain $D(\mathcal{B})$ of the operator $\mathcal{B}$ is given by $D(\mathcal{B})=\lbrace  U \in \mathcal{H} \ ;\mathcal{B}(U) \in\mathcal{H}\rbrace$. 
 We endow the Hilbert space $\mathcal{H}$ with the inner product  given by
\begin{equation}\begin{array}{lll}\label{prodct}
\displaystyle <U,V>_{\mathcal{H}}&=& \displaystyle\rho_1 \int_{\Omega} u_2 v_2 dx\  dy+\rho_2\int_{\Omega} u_4 v_4 dx dy+\rho_2 \int_{\Omega} u_6 v_6 dx \ dy\\ &+& \displaystyle D\int_{\Omega} u_{3_x} v_{3_x} dx \ dy +D\int_{\Omega} u_{5_y} v_{5_y} dx \ dy\\ &+&\displaystyle K \int_{\Omega} \left( u_3+u_{1_x}\right) \left( v_3+v_{1_x}\right)dx dy+ K \int_{\Omega} \left( u_5+u_{1_y}\right) \left( v_5+v_{1_y}\right)dx dy  \\
 &+&\displaystyle D\left(\frac{1-\mu}{2}\right) \int_{\Omega}\left( u_{3_y}+u_{5_x}\right) \left( v_{3_y}+v_{5_x}\right)dx dy +D\mu  \int_{\Omega} u_{3_x}v_{5_y} dx dy  
 \\
 &+&\displaystyle D\mu  \int_{\Omega} u_{5_y}v_{3_x} dx dy, 
\end{array}\end{equation}
and the norm 
\begin{eqnarray*}
\|U\|_{\mathcal{H}}^2&=&\displaystyle\rho_1\|u_2\|^2 + \rho_2 \|u_4\|^2 +\rho_2\|u_6\|^2 +D\|u_{3x}\|^2+ D\|u_{5y}\|^2+K\|u_3+u_{1_x}\|^2 \\&+& \displaystyle K  \|u_5+u_{1_y}\|^2+D \left(\frac{1-\mu}{2}\right)\|u_{3_y}+u_{5_x}\|^2
+ 2 D\mu \int_{\Omega} u_{3_x} u_{5_y} dx dy, 
\end{eqnarray*}
for all $U=(u_1,u_2,u_3,u_4,u_5,u_6)^T$ and $V=(v_1,v_2,v_3,v_4,v_5,v_6)^T$.\\
Using the Korn and the Poincar\'e inequalities,  we easily show that $\|.\|_{\mathcal{H}}$ is equivalent to the usual norm in $\mathcal{H}$.\\
Moreover,
 the energy  associated with the solution of the system \eqref{MTD1}--\eqref{MTD3} is defined  as follows:
\begin{equation}
\begin{array}{c}\label{en}
 E(t)=E(t,w,\psi,\varphi):=\displaystyle\frac{1}{2} \int_{\Omega} \big[  \rho_1 |w_t|^2+\rho_2 |\psi_t |^2+ \rho_2 |\varphi_t |^2 +K|\psi+w_x |^2+K |\varphi+w_y |^2  \\ + \displaystyle\ D|\psi_x |^2+D|\varphi_y |^2
 + D\left( \frac{1-\mu}{2}\right) |\psi_y+ \varphi_x |^2+ 2D \mu \psi_x \varphi_y \big] dx dy. 
\end{array}
\end{equation}
\begin{remark}
Note that the energy $E(t)$ defines a norm in $\mathcal{H}$ which is equivalent to $\|.\|_{\mathcal{H}}$.
\end{remark}
In the following theorem, we formulate our result on the existence and uniqueness of the solution of (\ref{MTD1})--(\ref{MTD3}) associated with \eqref{BD} and \eqref{IC-new}.
\begin{theorem}\label{thm1}
Assume that $(H_{0})$  is satisfied. Then, the operator $\mathcal{
A}+\mathcal{B}$ generates a continuous semi-group $(\mathcal{T}(t))_{t\geq 0}$ on $%
\mathcal{H}$. Moreover,  for all
initial data $U_{0}\in \mathcal{H}$, the Cauchy problem \eqref{sst} has a unique solution $U\in \mathcal{C}([0,\infty );\mathcal{H}),$ and 
 for all initial data $U_{0}\in D(\mathcal{A})$, the
solution $U\in L^{\infty }([0,\infty );D(\mathcal{A}))\cap W^{1,\infty
}([0,\infty );\mathcal{H}).$
\end{theorem} 
We first state and prove the following lemmas which will be useful to deduce the well-posedness result in Theorem  \ref{thm1}.
\begin{lemma}\label{lemmonoton}
For $U=(u_1,u_2,u_3,u_4,u_5,u_6) \in D(\mathcal{A})$, we have $(\mathcal{A}U,U)_{\mathcal{H}}=0$.
\end{lemma}
\begin{proof}
Let $U=(u_1,u_2,u_3,u_4,u_5,u_6) \in D(\mathcal{A})$. Then, we have 
$$
\begin{array}{lcl}
\displaystyle <\mathcal{A}U,U>_{\mathcal{H}}&=&\displaystyle < K\Delta u_1+Ku_{3_x}+Ku_{5_y},u_2>  \\ 
&+&\displaystyle <-Ku_{1_x}+D\left(u_{3_{xx}}+\left(\frac{1-\mu}{2} \right)u_{3_{yy}}\right)-Ku_3+D\left(\frac{1+\mu}{2}\right)u_{5_{xy}}, u_4>\\
&+& \displaystyle <-Ku_{1_y}+D\left(\frac{1+\mu}{2}\right)u_{3_{xy}}+ D\left(\frac{1-\mu}{2}\right)u_{5_{xx}}+Du_{5_{yy}}-Ku_5,u_6> \\&+&\displaystyle D<u_{4_x},u_{3_x}>+ D<u_{6_y},u_{5_y}>+K <u_4+u_{2_x},u_3+u_{1_x}> \\
 &+&\displaystyle K  <u_6+u_{2y},u_5+u_{1y}>+ D \left(\frac{1-\mu}{2}\right) <u_{4_y}+u_{6_x},u_{3_y}+u_{5_x}> \\
&+&\displaystyle D\mu<u_{4_x},u_{5_y}>  +D\mu<u_{3_x},u_{6_y}>,
\end{array} $$
where $<.,.>_{\mathcal{H}}$ is given by \eqref{prodct} and $<.,.>$ is the inner product in $L^2(\Omega)$.\\ 
Using integration by parts and the boundary conditions \eqref{BD}, we obtain $ \Re \left(<\mathcal{A}U,U>_{\mathcal{H}}\right)=0$ and we conclude that $\mathcal{A}$  is a dissipative operator.
\end{proof}
\begin{lemma}\label{lemsurjective}
The operator $\mathcal{I}-\mathcal{A}$ is surjective from $\mathcal{D}(\mathcal{A})$ onto $\mathcal{H}$.
\end{lemma}
\begin{proof}
Given  any $W=(w_{1},w_{2},w_{3},w_{4},w_{5},w_{6})^T\in \mathcal{H}$, there exists \\ $V=(v_{1},v_{2},v_{3},v_{4},v_{5},v_{6})^T\in D(\mathcal{A})$ satisfying 
\begin{equation}\label{sdf}
(\mathcal{I}-\mathcal{A})V=W,
\end{equation}
that is,
 \begin{equation*}
  \left\{
\begin{array}{l}
v_1-v_2=w_1,\\
\displaystyle v_2-\frac{k}{\rho_1}\Delta v_1-\frac{k}{\rho_1} v_{3x}-\frac{k}{\rho_1}v_{5y}=w_2, \\
v_3-v_4=w_3,\\
\displaystyle v_4+\frac{k}{\rho_2}v_{1x}-\mathcal{A}_1v_3-\frac{D}{\rho_2} \left(\frac{1+\mu}{2}\right)v_{5xy}=w_4,\\
v_5-v_6=w_5,\\
\displaystyle v_6+\frac{k}{\rho_2} v_{1y}-\frac{D}{\rho_2}\left(\frac{1+\mu}{2}\right)v_{3xy}-\mathcal{A}_2v_4=w_6.
   \end{array}
 \right.
 \end{equation*}
The above system can be rewritten as follows: 
 \begin{equation*}
  \left\{
\begin{array}{l}
v_2=v_1-w_1,\\
\displaystyle v_1-\frac{k}{\rho_1}\Delta v_1-\frac{k}{\rho_1} v_{3x}-\frac{k}{\rho_1}v_{5y}=w_2+w_1, \\
v_4=v_3-w_3,\\
\displaystyle v_3+\frac{k}{\rho_2}v_{1x}-\mathcal{A}_1v_3-\frac{D}{\rho_2}\left(\frac{1+\mu}{2}\right) v_{5xy}=w_4+w_3,\\
v_6=v_5-w_5,\\
\displaystyle v_5+\frac{k}{\rho_2} v_{1y}-\frac{D}{\rho_2}\left(\frac{1+\mu}{2}\right)v_{3xy}-\mathcal{A}_2v_4=w_6 +w_5.
   \end{array}
 \right.
 \end{equation*}
 Now, it suffices to solve the following system, for $v_1$, $v_3$ and $v_5$,
  \begin{equation} \label{lax}
  \left\{
\begin{array}{l}
\displaystyle v_1-\frac{k}{\rho_1}\Delta v_1 =w_2+w_1+\frac{k}{\rho_1} v_{3x}+\frac{k}{\rho_1}v_{5y}, \\
\displaystyle v_3-\mathcal{A}_1v_3=w_4+w_3-\frac{k}{\rho_2}v_{1x}+\frac{D}{\rho_2}\left(\frac{1+\mu}{2}\right) v_{5xy},\\
\displaystyle v_5=w_6 +w_5-\frac{k}{\rho_2} v_{1y} +\frac{D}{\rho_2}\left(\frac{1+\mu}{2}\right)v_{3xy}+\mathcal{A}_2v_4.
   \end{array}
 \right.
 \end{equation}
 Multiplying $\eqref{lax}_1$ by $v_1$,  $\eqref{lax}_2$ by $v_3$, $\eqref{lax}_3$ by $v_5$ and integrating over $\Omega$, we obtain the following  variational formulation:
\begin{equation*}
\begin{array}{lrl}
&&\displaystyle \int_{\Omega} |v_1|^2 dx dy -\frac{k}{\rho_1}\int_{\Omega} \Delta v_1 v_1 dx dy  - \frac{k}{\rho_1} \int_{\Omega} \Big( v_{3x} v_1  +v_{5y} v_1\Big) dx dy \\
&& \displaystyle+\int_{\Omega}|v_3|^2 dx dy -\int_{\Omega} \mathcal{A}_1 v_3 v_3dx dy +\frac{k}{\rho_2}\int_{\Omega}v_{1x} v_3 dx dy-\frac{D}{\rho_2}\left(\frac{1+\mu}{2}\right)\int_{\Omega} v_{5xy}v_3 dx dy   \\
&&\displaystyle \int_{\Omega} |v_5|^2 dx dy+ \frac{k}{\rho_2}\int_{\Omega} v_{1y}v_5 dx dy  -\frac{D}{\rho_2}\left(\frac{1+\mu}{2}\right)\int_{\Omega} v_{3xy}v_5 dx dy -\int_{\Omega } \mathcal{A}_2 v_4 v_5 dx dy  \\&&= \displaystyle \int_{\Omega} (w_2+w_1)v_1 dx dy+ \int_{\Omega}( w_4+w_3 ) v_3 dx dy +\int_{\Omega} (w_6+w_5)v_5 dx dy
\end{array} 
\end{equation*}
Using integration by parts and the boundary conditions  with respect to $x$ and $y$, the above variational formulation yields the following form $$a((v_1,v_3,v_5), (v_1,v_3,v_5))=l(v_1,v_3,v_5).$$
 Now, thanks to the  Lax-Milgram theorem (see \cite{B}), we deduce that the equation \eqref{sdf} admits a unique solution $V\in D(\mathcal{A}).$ Therefore $\mathcal{A}$ is a maximal operator.
\end{proof}
\begin{lemma}\label{lemlipschitz}
The operator $\mathcal{B}$ is  Lipschitz and continuous on $\mathcal{H}$.
\end{lemma}
\begin{proof}
For all $U=(w,w_t,\psi,\psi_t,\varphi,\varphi_t)^{T}\in H $, $\tilde{U}=(\tilde{w},\tilde{w}_t,\tilde{\psi},\tilde{\psi}_t,\tilde{\varphi},\tilde{\varphi}_t)^T\in \mathcal{H}$ and using the assumption $(H_0)$, we have 
\begin{eqnarray*}
\Vert \mathcal{B}U-\mathcal{B}\tilde{U}\Vert_{H}^2 & =&\rho_1\Vert \chi_1(\psi_t)-\chi_1(\tilde{\psi}_t)\Vert_{L^2(\Omega)}^2  +\rho_2\Vert \chi_2(\varphi_t)-\chi_2(\tilde{\varphi}_t)\Vert_{L^2(\Omega)}^2\\
 &\leq &
C_{\chi_1, \chi_2} \left[\Vert \psi_t-\tilde{\psi}_t\Vert_{L^2(\Omega)}^2  +\Vert \varphi_t-\tilde{\varphi}_t \Vert_{L^2(\Omega)}^2\right]\\
&\leq& C_{\chi_1, \chi_2} \Vert U-\tilde{U} \Vert_{H}^2,
\end{eqnarray*}
where $C_{\chi_1, \chi_2}>0$.\\
Consequently, the operator $\mathcal{B}$ is Lipschitz.
Furthermore, the condition $\chi_i(0) = 0,   i\in \lbrace 1,2 \rbrace$ gives the continuity of the operator $\mathcal{B}$.
\end{proof}

Now, let us summarize the previous results. Thanks to Lemma \ref{lemmonoton} we have that $\mathcal{A}
$ is a dissipative operator. From Lemma \ref{lemsurjective} we showed that the operator  $\mathcal{A}$ is maximal and Lemma \ref{lemlipschitz} gives that the operator $\mathcal{B}$ is  Lipschitz and continuous. Finally, Thanks to  the nonlinear Hille-Yosida theorem (\cite{B}), we deduce that the Cauchy problem 
\eqref{sst} has a unique solution $U$ with respect to the regularity of the initial data $U_0$.

\section{\textbf{General asymptotic behavior}}\label{s4}
In this section, we discuss an upper energy estimate of the energy $E(t)$, given by \eqref{en}, that show in particular the decay to zero as $t \to \infty$.\\
Throughout the rest of this article, we will denote by $C$ a generic positive constant whose value may change from line to another.
\subsection{Dissipation of the energy}
We recall that the energy associated with system \eqref{MTD1}--\eqref{MTD3} is
defined by: 
\begin{equation}
\begin{array}{lrl}
 E(t)=E(t,w,\psi,\varphi)&:=&\displaystyle\frac{1}{2} \int_{\Omega} \big[  \rho_1 |w_t|^2+\rho_2 |\psi_t |^2+ \rho_2 |\varphi_t |^2 +K|\psi+w_x |^2+K |\varphi+w_y |^2  \\ &+& \displaystyle\ D|\psi_x |^2+D|\varphi_y |^2
 + D\left( \frac{1-\mu}{2}\right) |\psi_y+ \varphi_x |^2+ 2D \mu \psi_x \varphi_y \big] dx dy. 
\end{array}
\end{equation} 
The dissipation relationship is given in the following proposition.
\begin{proposition}
Let $U=(w,w_t,\psi,\psi_t,\varphi,\varphi_t)^T $ be the solution of \eqref{MTD1}--\eqref{MTD3}. Then, the energy $ E(t)$ satisfies the following dissipation relationship:
\begin{equation}\label{dissip}
E^{\prime}(t)=-\int_{\Omega}\left( \psi_t \chi_1( \psi_t)+ \varphi_t \chi_2( \varphi_t) \right) dx dy , \hspace{0.5cm} \forall \ t\geq 0,
\end{equation}
and, hence, we have
$$E(t)\leq E(0), \hspace{0.3cm} \forall\  t\geq 0.$$
\end{proposition}
\begin{proof}
By multiplying the equations \eqref{MTD1}, \eqref{MTD2} and \eqref{MTD3}, respectively, by $w_t$, $\psi_t$ and $\varphi_t$, using the integration by parts, the boundary conditions \eqref{BD} and the assumption $(H_0)$, we obtain \eqref{dissip}.
\end{proof}
\subsection{Assumptions} We recall that $g$ is an odd and increasing function  ensuring thus the fulfillment of assumption $(H_0)$ in this manner, and we assume that the function $H$ satisfies the following properties:
 \begin{equation*} 
(H_1) \quad \left\{
\begin{array}{l} 
H(x)=\sqrt{x} g(\sqrt{x});\\
H  \ \text{is a strictly  convex function on} \  [0,r_0^2]  \ \text{for} \ r_0\in (0,1].
\end{array}
\right.
\end{equation*}%
Let  $\widehat{H}$ be a function  given by
\begin{equation*}\label{widH}
\widehat{H}(x):= \left\{
\begin{array}{l}
H(x), \  \mbox{if} \ x\in [0,r_0^2],\\
 \infty, \ \mbox{if} \ x\in \mathbb{R} \backslash [0,r_0^2].
\end{array}%
\right.
\end{equation*}
Next, we define the function $L$ on $[0,\infty)$ as
\begin{equation}\label{L}
L(y):= \left\{
\begin{array}{l}\frac{\widehat{H}^{\star}(y)}{y} \hspace{0,5cm} \mbox{if} \ y>0,\\
 0  \hspace{1,3cm} \mbox{if}\ y=0,
\end{array}%
\right.
\end{equation}
where $\widehat{H}^{\star}$ is the convex conjugate function of $\widehat{H}$ which is given by
\begin{equation*}
\widehat{H}^{\star}(y):=\sup_{x \in \mathbb{R}}\lbrace xy -\widehat{H}(x) \rbrace.
\end{equation*}
Then, we introduce  the function $\Lambda_{H}$ on $(0,r_0^2],$
\begin{equation*}\label{Lamda}
\Lambda_{H}(x):=\frac{H(x)}{xH'(x)},
\end{equation*}
 which is an essential tool to classify the feedback growth around $0$. To the best of our knowledge, this function has  been introduced for the first time in    \cite{R4} giving thus a  simplification of  the decay estimate formula obtained in \cite{AB1}.\\ 
 Finally, let $\eta >0$ and $M>0$ be two positive numbers and $f$ be a strictly increasing function from $[0,\eta)$ onto $[0,\infty)$. For any $r\in (0,\eta)$, we define the function  $K_r$ from $(0,r]$ on $[0,\infty)$ by
\begin{equation*}
K_r(\tau):=\int_{\tau}^r \frac{dy}{y f (y)},
\end{equation*}
and  a strictly increasing  function $\psi_{r}$ given  by,
\begin{equation*}
\psi_r(z)=z+K_r\left(f^{-1}\left( \frac{1}{z}\right)\right)\geq z, \quad \forall \ z\geq \frac{1}{f(r)},
\end{equation*}
where $f$ is a non-negative  and
strictly increasing $\mathcal{C}^1$-function defined from $[0,2\beta r_0^2)$ onto $[0,\infty)$ (see Remark \ref{rem4.1} below for the expression of $f$ and \cite{AB1} for the proof of its properties).\\
Now, for $x\geq \frac{1}{H'(r_0^2)}$, we define
\begin{equation*}
\psi_0(x):=\frac{1}{H'(r_0^2)}+\int_{1/x}^{H'(r_0^2)}\frac{1}{s^2(1-\Lambda_H((H')^{-1}(s)))} ds.
\end{equation*}
\begin{remark}\label{rem4.1}
 The function $L$, defined by \eqref{L}, is a strictly increasing and continuous function from $[0,\infty)$ onto $[0,r_0^2)$; see the proof in   \cite[Proposition 2.2]{AB1}.\\ We recall here the definition of the function $f$ given in \cite{AB2} by
 \begin{equation}\label{ffff1}
 f(s)=L^{-1}\left(\frac{s}{2\beta}\right), \quad \forall \; s\in [0,2\beta r_0^2),
 \end{equation}
 where  $\beta = \beta (E(0))$ is given by \eqref{beta} below.
\end{remark}

The first main result of this section is given by the following theorem. The method used for the proof is adapted from the general approach  in \cite{AB1} and applied here to the plate coupled system.
 \begin{theorem}\label{th2}
Assume that $v_1^2=v_2^2$ and that the assumptions $(H_0)$ and $(H_1)$ hold true. Moreover, suppose that the initial energy $E(0)$ and the parameter $\beta$ satisfy
\begin{equation}\label{bt1}
\frac{E(0)}{2L(H'(r_0^2))}\leq \beta.
\end{equation}
  Then, the total energy of \eqref{MTD1}--\eqref{MTD3}, which is defined by \eqref{en}, decays as
  \begin{equation}\label{energy-est}
 E(t)\leq  \beta L\left( \frac{1}{\psi_0^{-1}\left(\frac{t}{\sigma}\right)}\right), \quad \forall \ t\geq \frac{\sigma}{H'(r_0^2)}.
 \end{equation}
 Furthermore, if $\limsup_{x\rightarrow 0^{+}}\Lambda_H(x)<1,$ then $E$ satisfies  the following simplified decay rate
  \begin{equation*}\label{dcry}
 E(t)\leq \beta \left(H'\right)^{-1}\left( \frac{\sigma}{t}\right), \hspace{0.2cm}\forall  \ t\geq 0.
 \end{equation*}
Here, $\beta$ and $\sigma $ are, respectively, given by
 \begin{equation}\label{beta}
  \beta=\max \left( c_3, \frac{E(0)}{2L(H'(r_0^2))}\right),
   \end{equation}
   and
    \begin{equation}\label{M}
  \sigma = 2 \left(\alpha_3 \left( \frac{1}{c_3} +\frac{1}{c_6}  \right) \left( H'(r_0^2) \alpha_2 (c_4+c_5)+1\right)\right),
  \end{equation}
  where the $c_i$'s are introduced in \eqref{c3c4} and \eqref{c5c6}.
   \end{theorem}
   
   Before showing the proof of Theorem \ref{th2}, we state and prove the following proposition.
\begin{proposition}\label{lemmado}
 	 Assume that $v_1^2=v_2^2$ and $\chi_1$, $\chi_2$ satisfy $(H_0)$. Let $ f$ be a strictly increasing  function from $[0,\eta)$ onto $[0,\infty)$. Then, there exist positive constants $\alpha_i$ ($i=1,\ldots,3$) independent of $f$, such that the energy of the solution of \eqref{MTD1}--\eqref{MTD3} verifies the following nonlinear weighted estimate:
\begin{equation}
\begin{array}{lrl}\label{domenrg}
\displaystyle\int_S^T f(E(t))E(t) dt
&\leq &  \displaystyle\alpha_1 E(S) f(E(S))\\
& +&\displaystyle \alpha_2 \int_S^T f(E(t))\left(\int_{\Omega}|\chi_1(\psi_t)|^2 + |\chi_2(\varphi_t)|^2 dx dy\right) dt \\  && + \displaystyle\alpha_3 \int_S^T f(E(t)) \left(\int_{\Omega} |\psi_t|^2 +  |\varphi_t|^2 dx dy\right)  dt.
\end{array}	
 \end{equation}
\end{proposition}

The proof of Proposition \ref{lemmado} is carried out in Subsection \ref{proof prop4.3}. To this end, we first show some auxiliary results that we gather in the following subsection.
\subsection{Technical Lemmas}\mbox{}\\
The main purpose of this subsection is to make the presentation clear and organized for the reader so that the tedious computations for the energy estimates can be divided in many steps. Hence, in the following we state and prove some useful technical lemmas.
\begin{lemma}\label{l178}
Let $U=(w,w_t,\psi,\psi_t,\varphi,\varphi_t) $ be the solution of \eqref{MTD1}--\eqref{MTD3} with  the boundary conditions \eqref{BD} and we assume that $K>1$. Then for any   non-negative, $\mathcal{C}^1$ and strictly increasing  function $f$.
\begin{equation}
\begin{array}{lrl}\label{llkk}
&&\displaystyle (K-1)\int_S^T f(E(t)) \int_{\Omega} |\psi+w_x|^2 dx dy dt \\&&\displaystyle\leq \frac{1}{2}\int_{S}^T f(E(t))\int_{\Omega}|\chi_1(\psi_t)|^2 dx dy dt +\frac{1}{2}|\rho_2-\frac{\rho_1D}{K}|^2 \int_S f(E(t))|\psi_{tt}|^2 dx dy dt \\&& +\displaystyle \frac{\rho_1D}{K}\int_S^T \int_{\Omega}f(E(t))|\psi_t|^2 dx dy dt + cE(S)f(E(S)) \\ &&+\displaystyle D\left( \frac{1+\mu}{2} \right)\int_{S}^T f(E(t)) \int_{\Omega}  \varphi_{xy} (\psi+w_x)\ dx  dy   dt \\ &&+ \displaystyle D\left( \frac{1-\mu}{2} \right)\int_{S}^T f(E(t)) \int_{\Omega} \psi_{yy} (\psi+w_x)\  dx dy  dt \\ && \displaystyle  + D  \int_{S}^T f(E(t)) \int_{\Omega} (\varphi+w_y)_y \ \psi_x\ dx dy  dt  \\ && +\displaystyle \frac{ D}{2\varepsilon} \int_S^T f(E(t))\int_0^{L_2}\left[ \psi_x^2(L_1,y,t)+ \psi_x^2(0,y,t)\right]  dy  dt\\  
 &&  + \displaystyle \frac{ D\varepsilon}{2} \int_S^T f(E(t))\int_0^{L_2}\left[
  w_x^2(L_1,y,t)+ w_x^2(0,y,t) \right] dy  dt.
\end{array}	
 \end{equation}
\end{lemma}
\begin{proof}
First, we multiply the  equation  \eqref{MTD1} by $f(E(t)) \frac{D}{K}\psi_x$ and the  equation \eqref{MTD2} by  $f(E(t)) (\psi+w_x)$, then we integrate the resulting equations over $[S,T]\times \Omega$. Hence, we obtain, respectively,
\begin{equation}
\begin{array}{lrl}\label{aa1}
  &&\displaystyle\rho_1  \frac{D}{K} \int_S^T   f(E(t)) \int_{\Omega} w_{tt}  \psi_x  dx  dy \, dt - D \int_{S}^T f(E(t)) \int_{\Omega} (\psi+w_x)_x \ \psi_x  dx dy  dt \\ 
&&-\displaystyle D \int_{S}^T f(E(t)) \int_{\Omega} (\varphi+w_y)_y \ \psi_x   dx  dy  dt =0,
\end{array}	
 \end{equation}
and
\begin{equation}
\begin{array}{lrl}\label{aa22}	
&&\displaystyle \rho_2  \int_{S}^T f(E(t) \int_{\Omega} \psi_{tt} (\psi+w_x)  dx  dy  dt- D  \int_{S}^T f(E(t)) \int_{\Omega}\psi_{xx}(\psi+w_x) dx  dy  dt \\
&&-\displaystyle D \left(\frac{1-\mu}{2}\right)\int_{S}^T f(E(t)) \int_{\Omega} \psi_{yy} (\psi+w_x)  dx dy  dt  
\\  & &-\displaystyle D\left( \frac{1+\mu}{2} \right)\int_{S}^T f(E(t)) \int_{\Omega} \varphi_{xy} (\psi+w_x)  dx dy  dt  + K \int_{S}^T f(E(t) \int_{\Omega} |\psi+w_x|^2  dx  dy  dt \\
&&+  \displaystyle\int_{S}^T f(E(t) \int_{\Omega}  \chi_1(\psi_t)(\psi+w_x)   dx  dy   dt =0. 
\end{array}	
\end{equation}	
Second, summing the  equations \eqref{aa1} and \eqref{aa22} and using an integration by parts with respect to the variable $x$ in the term 
$ D  \int_{S}^T f(E(t) \int_{\Omega}\psi_{xx}(\psi+w_x) dx  dy  dt.
$
We arrive at
\begin{equation}
\begin{array}{lrl}\label{y1}
&&\displaystyle K\int_{S}^T f(E(t) \int_{\Omega} |\psi+w_x|^2  dx  dy  dt \\ &&=\displaystyle D \int_{S}^T f(E(t)) \int_{\Omega} (\varphi+w_y)_y \ \psi_x dx dy  dt  - \rho_2 \int_{S}^T f(E(t)  \int_{\Omega}  \psi_{tt} (\psi+w_x)  dx  dy  dt \\ && \displaystyle- \rho_1 \frac{D}{K} \int_S^T   f(E(t)) \int_{\Omega} w_{tt}  \psi_x   dx dy  dt  + D\left( \frac{1\!-\!\mu}{2} \right)\int_{S}^T \!f(E(t))  \int_{\Omega}  \psi_{yy} (\psi+w_x) dx dy dt\\
&&+ \displaystyle D\left( \frac{1\!+\!\mu}{2} \right)\int_{S}^T \!f(E(t)) \int_{\Omega}  \varphi_{xy} (\psi+w_x)  dx  dy  dt  - \int_{S}^T f(E(t) \int_{\Omega}  \chi_1(\psi_t)(\psi+w_x)  dx  dy  dt\\ && \displaystyle+D \int_S^T f(E(t))\int_0^{L_2} \left[\psi_x(L_1,y,t) w_x(L_1,y,t)- \psi_x(0,y,t) w_x(0,y,t) \right] dy  dt
.\end{array}	
\end{equation} 
Now, integrating by parts with respect to $x$ and  $t$  in some terms of the above  equality \eqref{y1}, we obtain
\begin{equation*}
\begin{array}{lcl}
&& \displaystyle K \int_{S}^T f(E(t)) \int_{\Omega} |\psi+w_x|^2 dx dy  dt   \\ &&= \displaystyle\left( \frac{\rho_1D}{K}-\rho_2\right) \int_S^T   f(E(t)) \int_{\Omega} \psi_{tt}  (w_x+\psi)  dx dy dt + \frac{\rho_1D}{K}\int_S^T f(E(t)) \int_{\Omega}|\psi_{t}|^2  dx dy  dt   \\ 
&&\displaystyle -\left[ \frac{\rho_1D}{K} f(E(t))\int_{\Omega} \psi_t \psi\right]_S^T
 +\frac{\rho_1D}{K} \int_S^T f'(E(t))E'(t)  \int_{\Omega} \psi_t \psi \  dx dy  dt  \\ &&\displaystyle -\frac{\rho_1D}{K} \int_S^T  f(E(t))
 \int_{\Omega} (w_t\psi_x+w_x\psi_t)_t \ dx dy dt  \\ && \displaystyle+\frac{\rho_1D}{K} \int_S^T  f(E(t))
  \int_0^{L_2}\left[ w_t(L_1,y,t)\varphi_t(L_1,y,t)-w_t(0,y,t)\varphi_t(0,y,t)\right]dy dt\\ &&\displaystyle +D \int_S^T f(E(t))\int_0^{L_2}\left[ \psi_x(L_1,y,t) w_x(L_1,y,t)- \psi_x(0,y,t) w_x(0,y,t) \right]dy  dt \\ &&
 +\displaystyle D\left( \frac{1+\mu}{2} \right)\int_{S}^T f(E(t)) \int_{\Omega}  \varphi_{xy} (\psi+w_x)\ dx  dy   dt \\ && +\displaystyle D\left( \frac{1-\mu}{2} \right)\int_{S}^T f(E(t)) \int_{\Omega} \psi_{yy} (\psi+w_x)\  dx dy  dt\\ && \displaystyle + D  \int_{S}^T f(E(t)) \int_{\Omega} (\varphi+w_y)_y \ \psi_x\ dx dy  dt - \int_{S}^T f(E(t)) \int_{\Omega}  \chi_1(\psi_t)(\psi+w_x)   dx  dy dt. 
\end{array}	
\end{equation*}
Since $\varphi=w=0$ at the boundary, then
\begin{equation*}
 \frac{\rho_1D}{K} \int_S^T  f(E(t))
  \int_0^{L_2} \left[ w_t(L_1,y,t)\varphi_t(L_1,y,t)-w_t(0,y,t)\varphi_t(0,y,t)\right]dy dt=0.
\end{equation*}
Now, using  the Young's inequality and for any $\varepsilon>0$, we  estimate the following term which actually appears because of the  non-vanishing boundary conditions
\begin{eqnarray*}
&&D \int_S^T f(E(t))\int_0^{L_2} \left[\psi_x(L_1,y,t) w_x(L_1,y,t)- \psi_x(0,y,t) w_x(0,y,t) \right]dy  dt\nonumber\\ && \leq  \frac{D}{2\varepsilon} \int_S^T f(E(t))\int_0^{L_2}\left[ \psi_x^2(L_1,y,t) + \psi_x^2(0,y,t)  \right]dy  dt \nonumber\\ && +\frac{D\varepsilon}{2}\int_S^T f(E(t))\int_0^{L_2}\left[ w_x^2(L_1,y,t) + w_x^2(0,y,t) \right] dy  dt \nonumber.
\end{eqnarray*}
Therefore, we have the following inequality: 
\begin{equation}
\begin{array}{lcl}\label{158}
&&\displaystyle K \int_{S}^T f(E(t)) \int_{\Omega} |\psi+w_x|^2 dx dy  dt\\ &&\displaystyle\leq  \left( \frac{\rho_1D}{K}-\rho_2\right) \int_S^T   f(E(t)) \int_{\Omega} \psi_{tt}  (w_x+\psi)  dx dy dt\\ &&+\displaystyle
\frac{\rho_1D}{K}\int_S^T f(E(t)) \int_{\Omega}|\psi_{t}|^2  dx dy  dt \\&& \displaystyle -\left[ \frac{\rho_1D}{K} f(E(t))\int_{\Omega} \psi_t \psi\right]_S^T
 +\frac{\rho_1D}{K} \int_S^T f'(E(t))E'(t)  \int_{\Omega} \psi_t \psi \  dx dy  dt \\ &&\displaystyle -\frac{\rho_1D}{K} \int_S^T  f(E(t))
 \int_{\Omega} (w_t\psi_x+w_x\psi_t)_t \ dx dy dt  
 \\ && \displaystyle+ D\left( \frac{1+\mu}{2} \right)\int_{S}^T f(E(t) \int_{\Omega}  \varphi_{xy} (\psi+w_x)\ dx  dy   dt \\ \displaystyle &&+\displaystyle D\left( \frac{1-\mu}{2} \right)\int_{S}^T f(E(t)) \int_{\Omega} \psi_{yy} (\psi+w_x)\  dx dy  dt  \\ &&\displaystyle + D  \int_{S}^T f(E(t)) \int_{\Omega} (\varphi+w_y)_y \ \psi_x\ dx dy  dt  - \int_{S}^T f(E(t)) \int_{\Omega}  \chi_1(\psi_t)(\psi+w_x)   dx  dy dt \\ && \displaystyle+ \frac{ D}{2\varepsilon} \int_S^T f(E(t))\int_0^{L_2} \left[\psi_x^2(L_1,y,t)+ \psi_x^2(0,y,t)\right]  dy  dt \\  
 &&\displaystyle +  \frac{ D\varepsilon}{2} \int_S^T f(E(t))\int_0^{L_2}\left[
  w_x^2(L_1,y,t)+ w_x^2(0,y,t)\right]  dy  dt.
\end{array}
\end{equation}

In the sequel, we will estimate some terms appearing in \eqref{158} using Young's inequality.  First, we see that
\begin{equation}
\begin{array}{lcl}\label{cc1}
\displaystyle\left\vert \int_S^T f(E(t))\int_{\Omega} \chi_1(\psi_t)(\psi+w_x)\right\vert\leq  \frac{1}{2}\int_S^T f(E(t))\int_{\Omega} (|\chi_1(\psi_t)|^2  + |\psi+w_x|^2) dx dy  dt,
\end{array}
\end{equation}
and
\begin{equation}
\begin{array}{lcl}\label{cc2}
&\displaystyle\left\vert \left(\rho_2-\frac{\rho_1D}{K}\right) \int_S^T   f(E(t)) \int_{\Omega} \psi_{tt}  (w_x+\psi) \ dx dy dt \right\vert\\ &\displaystyle\leq  \frac{1}{2} |\rho_2-  \frac{\rho_1D}{K}|^2 \int_S^T   f(E(t)) \int_{\Omega} |\psi_{tt}|^2 dx  dy  dt \\ & \displaystyle + \frac{1}{2}\int_S^T   f(E(t)) \int_{\Omega}  |w_x+\psi|^2  dx dy  dt,
\end{array}
\end{equation}
Then, we perform an integration by parts for the term (appearing in \eqref{158}),
 $$\displaystyle -\frac{\rho_1D}{K} \int_S^T  f(E(t))
 \int_{\Omega} (w_t\psi_x+w_x\psi_t)_t \ dx dy dt.$$   Since $E$ is a non-decreasing function, whereas $f$ is an increasing function,  we have 
\begin{equation}
\begin{array}{lcl}\label{cc3}
\displaystyle\left\vert\frac{\rho_1 D}{K}\left[ f(E(t))\int_{\Omega} \Big(\psi_t(\psi+w_x)+w_t\psi_x\Big) dxdy \right]_S^T\right\vert \leq &&\displaystyle\left\vert \left[ \frac{\rho_1 D}{K} f(E(t))E(t)\right]_S^T \right\vert \\  && \displaystyle\leq  c E(S)f(E(S)), \quad  \forall\  0 \leq S\leq T,
\end{array}
\end{equation}
where $c$ is a generic constant  which may depend on $L_1, L_2, \rho_1,  \rho_2, K$ and $D$.\\
Similarly for any  $ 0\leq S \leq T$, we estimate the following term as
\begin{equation}
\begin{array}{lcl}\label{cc4}
&&\displaystyle\left\vert\frac{\rho_1D}{K} \int_S^T f'(E(t))E'(t)\int_{\Omega}\Big(\psi_t(\psi+w_x)-\psi_t w_x  \Big)dx dy  \, dt\right\vert\\ &\leq& \displaystyle c E(S) \int_S^T f'(E(t))(-E'(t))E(t)dt\leq c E(S)f(E(S)).
\end{array}
\end{equation}
We inject  the estimates \eqref{cc1}, \eqref{cc2}, \eqref{cc3} and \eqref{cc4} in \eqref{158}, and recall that $K>1$,  we end up with the estimate \eqref{llkk}.
\end{proof}

\begin{remark}
From the physical point of view, the shear modulus $K$ serves as a measure of a material's resistance to shear deformation. For common materials such as Aluminum, Steel, Copper, and others, $K$ is greater than $1$. Therefore, the assumption  $K>1$ is naturally applicable to the specific type of solid under consideration in this paper. More specifically, structural steel is chosen for its high shear modulus, allowing it to resist deformation caused by shear forces. In the context of building construction, external factors like wind loading can induce twisting stresses on the structure, resulting in shear stresses within the steel components.
\end{remark}

Now, we state and prove the following lemmas which show the control of the boundary terms appearing in \eqref{llkk}.
 
\begin{lemma} \label{l74}
Let $U=(w,w_t,\psi,\psi_t,\varphi,\varphi_t) $ be the solution of \eqref{MTD1}--\eqref{MTD3} with  the boundary conditions \eqref{BD}. Then, for any   non-negative and strictly increasing  $\mathcal{C}^1-$function $f$, we have, for all  $S \leq t \leq T,$
\begin{equation}
\begin{array}{lcl} \label{jht}
&&\displaystyle D\left(\int_S^Tf(E(t)) \int_0^{L_2} \left[ \psi_x^2(L_1,y,t) + \psi_x^2(0,y,t) \right]dy  \right)\\
&& \leq \displaystyle \kappa_1 E(S) f(E(S)) + \kappa_2   \left(1+\frac{1}{\eta}\right)\int_S^{T}f(E(t)) \int_{\Omega} \psi_x^2 dx dy dt \\ && + \displaystyle\eta \int_S^T f(E(t))\int_{\Omega} |\psi+w_x|^2 dx dy dt +\kappa_1\int_S^T f(E(t))\int_{\Omega} \left(|\chi_1(\psi_t)|^2+|\psi_t|^2\right) dx dy dt \\ &&+\displaystyle \kappa_2 \left[D\left(\frac{1-\mu}{2}\right)\int_S^{T}f(E(t))\int_{\Omega} \psi_{yy}\psi_x dx dy dt \right]\\ && +\displaystyle \kappa_3  \left[ D\left(\frac{1+\mu}{2}\right)\int_S^Tf(E(t))\int_{\Omega} \varphi_{xy}\psi_x dx dy dt\right],
\end{array}
\end{equation}
where $\kappa_i$, $i=1,\ldots,3$, does not depend on $\eta$.
\end{lemma}
\begin{proof}
Let $q_1$ be a  $\mathcal{C}^1([0,L_1])-$function  such that $q_1(0)=-q_1(L_1)=2$. Then, multiplying the equation \eqref{MTD2} by $f(E(t))q_1\psi_x$, we obtain 
\begin{eqnarray}
&&\rho_2\int_S^T f(E(t))\int_{\Omega} \psi_{tt}q_1\psi_x dx dy dt -D\int_S^T f(E(t))\int_{\Omega} \psi_{xx}q_1\psi_x dx dy dt \nonumber\\&& - D\left(\frac{1-\mu}{2}\right)\int_S^T f(E(t))\int_{\Omega} \psi_{yy} q_1 \psi_x dx dy dt\\ && - D\left(\frac{1+\mu}{2}\right) \int_S^T f(E(t))\int_{\Omega} \varphi_{xy} q_1 \psi_x dx dy dt \nonumber\\&&+\int_S^T f(E(t))\int_{\Omega} K(\psi+w_x)q_1\psi_x  dx dy dt +\int_S^T f(E(t))\int_{\Omega} \chi_1(\psi_t)q_1\psi_x=0 \nonumber.
\end{eqnarray}
Then, integrating by parts with respect to the time variable  $t$ and the space variable $x$ and using the fact that $\psi=0$ at the boundary, we find
\begin{eqnarray*}
&&\rho_2\int_S^T f(E(t))\int_{\Omega }q_{1x}\frac{ \psi_t^2}{2} dx dy dt + \rho_2 \left[f(E(t))\int_{\Omega} q_1 \psi_x \psi_t dx dy \right]_S^{T}\\ &&- \rho_2 \int_S^T f'(E(t))E'(t)\int_{\Omega} q_1\psi_x\psi_t dx dy dt \nonumber\\ && -\frac{D}{2}\left(\int_S^Tf(E(t))\int_0^{L_2} \Big(q_1(L_1) \psi_x^2(L_1,y,t) -q_1(0)\psi_x^2(0,y,t)\Big) dy dt\right)\\ &&+ D\int_S^Tf(E(t))\int_{\Omega}q_{1x} \frac{\psi_x^2}{2} dx dy dt \nonumber\\ &&
-D\left(\frac{1-\mu}{2}\right)\int_S^T f(E(t))\int_{\Omega} \psi_{yy}q_1\psi_x  dx dy dt\\ && -D\left(\frac{1+\mu}{2}\right)\int_S^T f(E(t))\int_{\Omega} \varphi_{xy}q_1 \psi_x dx dy dt \nonumber\\ &&
+K\int_S^T f(E(t))\int_{\Omega} (\psi+w_x)q_1 \psi_x dx dy dt+ \int_S^T f(E(t))\int_{\Omega}  \chi_1(\psi_t)q_1 \psi_xdx dy dt=0. \nonumber
\end{eqnarray*}
Employing $q_1(0)=2=-q_1(L_1)$, we infer that
\begin{equation*} 
\begin{array}{lcl}
& & \displaystyle D \left(\int_S^Tf(E(t)) \int_0^{L_2} \left[ \psi_x^2(L_1,y,t) + \psi_x^2(0,y,t) \right]dy  \right) \\ && \displaystyle=\rho_2 \int_S^T f'(E(t))E'(t)\int_{\Omega} q_1\psi_x\psi_t dx dy dt -\rho_2 \left[f(E(t))\int_{\Omega} q_1 \psi_x \psi_t dx dy \right]_S^{T} \\ &&\displaystyle -\rho_2 \int_S^T f(E(t))\int_{\Omega} q_{1x}\frac{\psi_t^2 }{2}dx dy dt   - D\int_S^Tf(E(t))\int_{\Omega}q_{1x} \frac{\psi_x^2}{2} dx dy dt \\ && 
+\displaystyle D\left(\frac{1-\mu}{2}\right)\int_S^T f(E(t))\int_{\Omega} \psi_{yy}q_1\psi_x  dx dy dt\\ && \displaystyle+D\left(\frac{1+\mu}{2}\right)\int_S^T f(E(t))\int_{\Omega} \varphi_{xy}q_1 \psi_x dx dy dt \\ &&
\displaystyle-K\int_S^T f(E(t))\int_{\Omega} (\psi+w_x)q_1 \psi_x dx dy dt - \int_S^T f(E(t))\int_{\Omega}  \chi_1(\psi_t)q_1 \psi_x dx dy dt,\\ && \quad  \forall \ S \leq t \leq T.
\end{array}
\end{equation*}
Using the fact that $q_1$ is a bounded function, denoting $\kappa=max_{x\in[0,L_1]}|q_1(x)|$  and using the Young's inequality, we can easily  estimate the term  $$\displaystyle K\int_S^T f(E(t))\int_{\Omega} (\psi+w_x)q_1 \psi_x dx dy dt.$$ 

Hence, we obtain for all $\eta >0$ the desired result \eqref{jht}.
\end{proof}
\begin{lemma}\label{l75}
Let $U=(w,w_t,\psi,\psi_t,\varphi,\varphi_t) $ be the solution of \eqref{MTD1}--\eqref{MTD3} with  the boundary conditions \eqref{BD}. Then, for any   non-negative and strictly increasing  $\mathcal{C}^1-$function $f$, we have, for all  $S \leq t \leq T$,
\begin{equation}
\begin{array}{lcl}\label{gte}
&&\displaystyle\int_S^Tf(E(t)) \int_0^{L_2} \left[ w_x^2(L_1,y,t) + w_x^2(0,y,t) \right]dy  \\&& \leq c \displaystyle E(S)f(E(S))+c \int_S^T f(E(t))\int_{\Omega}\left(w_t^2+w_x^2+\psi_x^2 + |\varphi_y+w_{yy}|^2\right)dx dy dt,   
\end{array}
\end{equation}
where $c$ is a positive constant.
\end{lemma}
\begin{proof}
Let $q_1$ be a  $\mathcal{C}^1([0,L_1])-$function  such that $q_1(0)=-q_1(L_1)=2$.\\
First, we multiply the equation \eqref{MTD1} by $f(E(t))q_1 w_x$ and we integrate with respect to time ($t$) and space ($x$) variables. Hence, we obtain, for all $ S \leq t \leq T,$
\begin{equation*}
\begin{array}{lcl}\label{llln}
&& \displaystyle\left(\int_S^Tf(E(t)) \int_0^{L_2}\left[  w_x^2(L_1,y,t) + w_x^2(0,y,t) \right]dy  \right)\\ &&=\displaystyle\rho_1 \int_S^T f'(E(t))E'(t)\int_{\Omega} q_1w_x w_t dx dy dt -\rho_2 \left[f(E(t))\int_{\Omega} q_1 w_x w_t dx dy \right]_S^{T} \\ &&\displaystyle - K\int_S^Tf(E(t))\int_{\Omega}q_{1x} \frac{w_x^2}{2} dx dy dt  
+K\int_S^T f(E(t))\int_{\Omega} \psi_x q_1 w_x  dx dy dt  \\ &&+\displaystyle K\int_S^T f(E(t))\int_{\Omega} (\varphi+w_y)_y q_1 w_x  dx dy dt-\rho_1\int_S^T f(E(t))\int_{\Omega} q_{1_x}\frac{w_t^2}{2}  dx dy dt. 
\end{array}
\end{equation*}
Using the fact that $q_1$ is a bounded function, recalling the definition $$\kappa=max_{x\in[0,L_1]}|q_1(x)|,$$  and utilizing the Poincar\'e  inequality for $\psi$ to estimate the term  $$ \displaystyle K\int_S^T f(E(t))\int_{\Omega} \psi_x q_1 w_x,$$ and the Young inequality, we end up with estimating the term $$\displaystyle \rho_2 \int_S^T f'(E(t))E'(t)\int_{\Omega} q_1w_xw_t dx dy dt,$$ and 
 $\displaystyle \rho_1 \left[f(E(t))\int_{\Omega} q_1 w_x w_t dx dy \right]_S^{T}$. This concludes the desired result \eqref{gte}.
 \end{proof}
\begin{lemma}\label{l1888}
Let $U=(w,w_t,\psi,\psi_t,\varphi,\varphi_t) $ be the solution of \eqref{MTD1}--\eqref{MTD3} with  the boundary conditions \eqref{BD}. Assume that $K>1$. Then, for any   non-negative  and strictly increasing  $\mathcal{C}^1-$function $f$, we have
\begin{equation*}
\begin{array}{lcl}\label{bgk}
&&C\displaystyle\int_{S}^T f(E(t)) \int_{\Omega} |\psi+w_x|^2  dx  dy  dt\\ && \displaystyle\leq C\displaystyle \int_{S}^T f(E(t)) \int_{\Omega} | \chi_1(\psi_t)|^2 dx dy  \, dt\\ && + \displaystyle\frac{1}{2}   \left(\frac{\rho_1D}{K}-\rho_2\right)^2  \int_S^T   f(E(t)) \int_{\Omega} |\psi_{tt}|^2  dx  dy \, dt \\
 &&+C\displaystyle\int_S^T f(E(t))\int_{\Omega}|\psi_{t}|^2+ |\psi_x|^2 dx dy dt  + C E(S)f(E(S)) \\ && 
  +\displaystyle \int_{S}^T f(E(t)) \int_{\Omega}   (\psi+w_x+\psi_x) \left(D\left( \frac{1-\mu}{2}
   \right)\psi_{yy} +  D\left( \frac{1+\mu}{2} \right)\varphi_{xy} \right) dx dy \,dt\\
  &&+ \displaystyle D  \int_{S}^T f(E(t)) \int_{\Omega} (\varphi+w_y)_y \ \psi_x  dx dy\, dt,
\end{array}
\end{equation*}
where $C$ is a positive  constant which depends on $K, \eta, \kappa, \kappa_i; i=\lbrace1,\ldots,3\rbrace$.
\end{lemma}
\begin{proof}
Combining the estimates \eqref{llkk}, \eqref{jht} and \eqref{gte} yields the desired result.
\end{proof}

\begin{lemma} \label{0001}
Let $U=(w,w_t,\psi,\psi_t,\varphi,\varphi_t) $ be the solution of \eqref{MTD1}--\eqref{MTD3} with  the boundary conditions \eqref{BD}. Then, for any   non-negative  and strictly increasing  $\mathcal{C}^1-$function $f$, we find that
\begin{equation}
\begin{array}{lcl}\label{BB}
 &&(K-1) \displaystyle\int_{S}^T f(E(t)) \int_{\Omega} |\varphi+w_y|^2  dx  dy  dt\\ &&\leq  c\displaystyle \int_S^T f(E(t))\int_{\Omega} |\chi_2(\varphi_t)|^2 dx dy dt + cE(S)f(E(S))  \\ && \displaystyle+\frac{1}{2}  \left(\rho_2 - \frac{\rho_1D}{K}\right)^2  \int_S^T   f(E(t)) \int_{\Omega} |\varphi_{tt}|^2 \ dx  dy dt  
 \\&&\displaystyle+ c\int_S^T f(E(t))\int_{\Omega}\Big(|\varphi_{t}|^2 +|\varphi_y|^2\Big) dx  dy  dt \\  &&\displaystyle +\int_{S}^T f(E(t)) \int_{\Omega}  (\varphi+w_y+\varphi_y)\left( D\left( \frac{1-\mu}{2} \right)\varphi_{xx} + D\left( \frac{1+\mu}{2} \right)\psi_{xy} \right) dx dy dt \\
  &&+ \displaystyle D \int_{S}^T f(E(t)) \int_{\Omega} (w_t^2+w_y^2+\varphi_y^2+|(\psi+w_x)_x|^2)  dx dy dt . 
\end{array}
\end{equation}
\end{lemma}

\begin{proof}
 We multiply the  equations  \eqref{MTD1} and  \eqref{MTD3}, respectively,  by $f(E(t)) \frac{D}{K}\varphi_y$ and $f(E(t))  (\varphi+w_y)$, then we integrate the resulting equations over $[S,T]\times$ $\Omega$ as follows:
\begin{equation}
\begin{array}{lcl}\label{eq105}
 &\displaystyle \rho_1 \frac{D}{K} \int_S^T   f(E(t)) \int_{\Omega} w_{tt}  \varphi_y  dx dy  dt - D \int_{S}^T f(E(t)) \int_{\Omega} (\psi+w_x)_x \ \varphi_y\  dx  dy dt &\\
 &\displaystyle  - D  \int_{S}^T f(E(t)) \int_{\Omega} (\varphi+w_y)_y \ \varphi_y  dx dy dt =0,&
\end{array}
\end{equation}
  and
\begin{equation}
\begin{array}{lcl}\label{eq106}
 	&&\displaystyle\rho_2  \int_{S}^T f(E(t)) \int_{\Omega} \varphi_{tt} (\varphi+w_y) \ dx \ dy \ dt+ K \int_{S}^T f(E(t)) \int_{\Omega} |\varphi+w_y|^2 dx  dy  dt \\
 &&- D \displaystyle \left(\frac{1-\mu}{2}\right)\int_{S}^T f(E(t)) \int_{\Omega} \varphi_{xx} (\varphi+w_y)  dx dy  dt
 \\ &&- \displaystyle D \left(\frac{1+\mu}{2}\right) \int_{S}^T f(E(t)) \int_{\Omega}  \psi_{xy} (\varphi+w_y)dx dy  dt    \\&&\displaystyle+ \int_{S}^{T}f(E(t)) \int_{\Omega} \chi_2(\varphi_t)(\varphi+w_y)dx dy  dt \\ &&- \displaystyle D  \int_{S}^T f(E(t)) \int_{\Omega}\varphi_{yy}(\varphi+w_y) dx dy  dt=0. 
\end{array}
\end{equation}
Then, performing similar  computations as in Lemma \ref{l178}, adding the equations  \eqref{eq105} and \eqref{eq106}, and integrating by parts with respect to time and space,  the following  boundary term appears in the computations
\begin{equation}
D \int_S^T f(E(t))\int_0^{L_1} \left[\varphi_y(x,L_2,t) w_y(x,L_2,t)- \varphi_y(x,0,t) w_y(x,L_2,t) \right]dx  dt.
\end{equation}
In the purpose to estimate the above quantity, we use similar calculations as in the proof of Lemmas \ref{l74} and  \ref{l75}, namely we set the function $q_2\in \mathcal{C}^1([0,L_2])$  satisfying 
$q_2(0)=-q_2(L_2)=2$ and we denote by $\theta=max_{y\in [0,L_2]}|q_2(y)| $. Next, we multiply the equation  \eqref{MTD3} by $f(E(t))q_2\varphi_y$ and \eqref{MTD1} by $f(E(t))q_2 w_y$, 
\begin{equation}
\begin{array}{lcl}\label{oo}
 &&\displaystyle\rho_2\int_S^T f(E(t))\int_{\Omega} \varphi_{tt} q_2(x)\varphi_y dx dy dt -D\int_S^T f(E(t))\int_{\Omega} \varphi_{yy}q_2 \varphi_y dx dy dt  \\ &&\displaystyle -D\left(\frac{1-\mu}{2}\right)\int_S^T f(E(t))\int_{\Omega} \varphi_{xx}q_2 \varphi_y dx dy dt\\&&\displaystyle +K\int_S^T f(E(t))\int_{\Omega}(\varphi+w_y)q_2 \varphi_y dx dy dt \\ &&\displaystyle +\int_S^Tf(E(t))\int_{\Omega}\chi_2(\varphi_t)q_2\varphi_y dx dy dt-D\left(\frac{1+\mu}{2}\right)\int_S^T f(E(t))\int_{\Omega} \psi_{xy}q_2\varphi_y dx dy dt =0,
 \end{array}
 \end{equation}
 and
\begin{equation}
\begin{array}{lll}\label{ooo1}
 &&\displaystyle \rho_1 \int_S^T f(E(t)) \int_{\Omega} w_{tt} q_2w_y dx dy-K\int_S^T f(E(t))\int_{\Omega} (\psi+w_x)_x q_2 w_y  \\ &&\displaystyle-K\int_S^T f(E(t)) \int_{\Omega} (\varphi+w_y)_y q_2 w_y=0.
 \end{array}
 \end{equation}
 Therefore, we start by the equation \eqref{oo} where we integrate by parts with respect to the time variable  $t$ and the space variable $x$. Then, using the fact that $\varphi=0$ at the boundary, we find
\begin{equation*}
\begin{array}{lcl}
&&\displaystyle\rho_2 \int_S^T f(E(t))\int_{\Omega}\frac{\varphi_t^2}{2}q_{2y} dx dy dt+\rho_2\left[f(E(t))\int_{\Omega} q_2\varphi_y \varphi_t \right]_S^T\\ &&\displaystyle+K\int_S^Tf(E(t))\int_{\Omega} (\varphi+w_y)q_2 \varphi_y dx dy dt \\ && -\displaystyle\rho_2\int_S^T f'(E(t))E'(t)\int_{\Omega} q_2\varphi_y\varphi_t dx dy dt +D\int_S^T f(E(t))\int_{\Omega} q_{2y}\frac{\varphi_y^2}{2} dx dy dt\\ && \displaystyle-D\left(\frac{1-\mu}{2}\right) \int_{\Omega}\varphi_{xx}q_2 \varphi_y dx dy dt\\ &&\displaystyle -\frac{D}{2}\int_S^T f(E(t))\int_0^1\left(q_2(L_2)\varphi_y^2(x,L_2,t)-q_2(0)\varphi_y^2(x,0,t)\right)dxdt\\ &&\displaystyle -D\left(\frac{1+\mu}{2}\right)\int_S^T f(E(t))\int_{\Omega} \psi_{xy} q_2 \varphi_y dx dy dt +\int_S^T f(E(t))\int_{\Omega}\chi_2(\varphi_t)q_2\varphi_y dx dy dt=0.
 \end{array}
 \end{equation*} 
Employing the fact that $q_2$ is a bounded function by $\theta$ and that  $q_2(0)=2=-q_2(L_2)$, we can easily  estimate the terms \\ $ \displaystyle \rho_2 \int_S^T f'(E(t))E'(t)\int_{\Omega} q_2\varphi_y\varphi_t dx dy dt,$ and $\displaystyle \rho_2 \left[f(E(t))\int_{\Omega} q_2 \varphi_y \varphi_t dx dy \right]_S^{T}$.\\ Finally, we use the Young inequality to estimate the term 
$$\displaystyle K\int_S^Tf(E(t))\int_{\Omega} (\varphi+w_y)q_2 \varphi_y dx dy dt.$$
 Hence, we obtain for all $\eta_0>0$ the following estimate:
\begin{equation}
\begin{array}{lcl}\label{3332}
&&  \displaystyle D\int_S^T f(E(t))\int_0^{L_1}\left(\varphi_y^2(x,L_2,t)+\varphi_y^2(x,0,t)\right)dx dy dt  \\ &&  \displaystyle\leq c E(S)f(E(S))+\left(1+\frac{1}{\eta_0}\right)\int_S^T f(E(t))\int_{\Omega} \varphi_y^2 dx dy dt \\&& + \displaystyle\eta_0 \int_S^T f((E(t))\int_{\Omega}|\varphi+w_y|^2 dx dy dt +\theta_1 \int_S^T f(E(t))\int_{\Omega} \Big(|\chi_2(\varphi_t)|^2+\varphi_t^2\Big) dx dy dt  \\&& + \displaystyle \theta_2 D\left(\frac{1-\mu}{2}\right)\int_S^T f(E(t))\int_{\Omega}\varphi_{xx}\varphi_y dx dy dt \\&& + \displaystyle \theta_3 D\left( \frac{1+\mu}{2}\right)\int_S^T f(E(t))\int_{\Omega} \psi_{xy}\varphi_y dx dy dt,
 \end{array}
 \end{equation} 
where  $\theta_i$  ($i=1,\ldots, 3$) are  positive constants independent of $\eta_0$ .\\
Furthermore, we perform some estimates for the terms in \eqref{ooo1} in which we utilize the properties of the function $q_2$, and we obtain 
\begin{equation}
\begin{array}{lcl}\label{3331}
&&\displaystyle\int_S^T f(E(t))\int_0^{L_1}\left(w_y^2(x,L_2,t)+w_y^2(x,0,t)\right)dx dt  \leq cE(S)f(E(S))\\&&\displaystyle+\int_S^Tf(E(t))\int_{\Omega} \Big(w_t^2+w_y^2+\varphi_y^2+|(\psi+w_x)_x|^2\Big) dx dy dt. 
 \end{array}
 \end{equation}
Thanks to  \eqref{3332} and \eqref{3331}, we deduce the inequality \eqref{BB}.
\end{proof} 		

\subsection{Proof of Proposition \ref{lemmado}}\label{proof prop4.3}
Now we are able to prove Proposition \ref{lemmado}. 
\begin{proof}
First, thanks to the results obtained in Lemmas \ref{l75} and  \ref{0001}, we deduce that
\begin{equation}
\begin{array}{lcl}\label{4441}
&&\quad \displaystyle C\int_{S}^T f(E(t)) \int_{\Omega}\left( |\psi+w_x|^2 +  |\varphi+w_y|^2\right) dx  dy  dt  \\&& \displaystyle\leq c \int_{S}^T f(E(t)) \int_{\Omega} \left(| \chi_1(\psi_t)|^2 + |\chi_2(\varphi_t)|^2\right) dx dy  \, dt \\ && + \displaystyle\frac{1}{2}   |  \frac{\rho_1D}{K}-\rho_2|^2  \int_S^T   f(E(t)) \int_{\Omega}\left( |\psi_{tt}|^2  + |\varphi_{tt}|^2\right)  dx  dy \, dt   \\
 &&\displaystyle\quad+c\int_S^T f(E(t))\int_{\Omega}\left(|\psi_{t}|^2+ |\psi_x|^2 +|\varphi_{t}|^2 +|\varphi_y|^2\right) dx  dy  dt + c E(S)f(E(S)  \\ && \displaystyle
  \quad + \int_{S}^T f(E(t)) \int_{\Omega}   (\psi+w_x+\psi_x) \left(D\left( \frac{1-\mu}{2}
   \right)\psi_{yy} +  D\left( \frac{1+\mu}{2} \right)\varphi_{xy} \right) dx dy \,dt \\
  &&\displaystyle \quad+\int_{S}^T f(E(t)) \int_{\Omega}  (\varphi+w_y+\varphi_y)\left( D\left( \frac{1-\mu}{2} \right)\varphi_{xx} + D\left( \frac{1+\mu}{2} \right)\psi_{xy} \right) dx dy dt \\
  &&\displaystyle\quad + D  \int_{S}^T f(E(t)) \int_{\Omega} (\varphi+w_y)_y \ \psi_x  dx dy\, dt \\ && \displaystyle  
+  D \int_{S}^T f(E(t)) \int_{\Omega} (\psi+w_x)_x\ \varphi_{y}  dx dy dt .
 \end{array}
 \end{equation}
On the one hand,  we multiply the   equation \eqref{MTD1} by $f(E(t)) w $ and the equation  \eqref{MTD2}  by $f(E(t)) \psi $, then  we integrate over $[S,T]\times\Omega$ and we obtain
\begin{equation}
\begin{array}{lcl}\label{eqat1}
 	&&\displaystyle\rho_1\left[  f(E(t))\int_{\Omega} w_{t}w    \ dx \ dy \right]_S^T
 	-\rho_1  \int_S^T f'(E(t)) E'(t)
 	\int_{\Omega}w w_t  dx dy dt \\
&& - K \displaystyle\int_S^T f(E(t))\int_{\Omega} (\psi+w_x)_x w \ dx\ dy \ dt -\rho_1 \int_S^T f(E(t))\int_{\Omega} |w_t|^2   dx  dy  dt \\
&& -\displaystyle K \int_S^T f(E(t))
 \int_{\Omega} (\varphi+w_y)_y w  dx  dy  dt =0,
 \end{array}
 \end{equation}
  and
\begin{equation}
\begin{array}{lcl}\label{eqat2}
 &&\displaystyle \int_{S}^T f(E(t) \int_{\Omega}  \chi_1(\psi_t)\psi   dx  dy dt - \rho_2 \int_{S}^T f'(E(t)) E'(t) \int_{\Omega}\psi_t \psi  dx  dy   dt
 \\ &&-\displaystyle\rho_2 \int_{S}^T f(E(t)) \int_{\Omega} |\psi_t|^2   dx  dy  dt
 	-D \left(\frac{1-\mu}{2}\right)\int_{S}^T f(E(t) \int_{\Omega} \psi_{yy} \psi  dx  dy  dt \\ &&-\displaystyle D\left( \frac{1+\mu}{2} \right)\int_{S}^T f(E(t) \int_{\Omega}  \varphi_{xy} \psi dx  dy  dt  - D \int_{S}^T f(E(t)) \int_{\Omega} \psi_{xx}\psi dx dy dt  \\
 	&& \displaystyle+K \int_{S}^T f(E(t)) \int_{\Omega} (\psi+w_x)\psi dx dy dt  + \rho_2 \left[ f(E(t))\int_{\Omega} \psi \psi_t dx dy\right]_S^T =0.
 \end{array}
 \end{equation}
On the other hand,  multiplying the  equation  \eqref{MTD3} by $f(E(t)) \varphi$ and  integrating over $[S,T]\times\Omega$ yield
\begin{equation}
\begin{array}{lcl}\label{eqat3}
 	 &&-\displaystyle\rho_2 \int_{S}^T f(E(t)) \int_{\Omega} |\varphi_t|^2 dx dy  dt+ \rho_2\left[  f(E(t))\int_{\Omega} \varphi_{t}\varphi    dx  dy \right]_S^T  \\
 	 &&-\displaystyle\rho_2\int_S^T  f'(E(t)) E'(t)\int_{\Omega} \varphi_{t}\varphi    dx  dy  dt -D  \left(\frac{1-\mu}{2}\right) \int_{S}^T f(E(t)) \int_{\Omega}\varphi_{xx} \varphi  dx  dy  dt  \\ &&- \displaystyle D   \int_{S}^T f(E(t)) \int_{\Omega} \varphi_{yy} \varphi  dx  dy dt- D \left(\frac{1+\mu}{2}\right)\int_{S}^T f(E(t)) \int_{\Omega}  \psi_{xy} \varphi  dx  dy  dt \\ &&+\displaystyle K \int_{S}^T f(E(t)) \int_{\Omega} (\varphi+w_y)\varphi  dx dy dt
 	+\int_{S}^T f(E(t)) \int_{\Omega} \chi_2(\varphi_t)\varphi dx dy dt =0.
 \end{array}
 \end{equation}
 Next, by summing \eqref{eqat1}, \eqref{eqat2} and \eqref{eqat3},  we find
\begin{equation}
\begin{array}{lcl}\label{19}
 &&\displaystyle\int_{S}^T  f(E(t)) \int_{\Omega} \left( \rho_1 |w_t|^2 \!\!+ \!\! \rho_2  |\varphi_t|^2  + \rho_2 |\psi_t |^2\right)  dx dy  dt\\ && =\displaystyle
- \rho_2 \int_{S}^T f'(E(t)) E'(t) \int_{\Omega}\psi_t \psi  \, dx  dy  dt\\
&+&\displaystyle \!\! \rho_1\left[  f(E(t))\int_{\Omega} w_{t}w  \right]_S^T-\rho_1  \int_S^T f'(E(t)) E'(t)\int_{\Omega}w w_t  \, dx dy dt  \\ &-& \displaystyle \!\!
\rho_2\int_S^T \!\!  f'(E(t)) E'(t)\int_{\Omega} \varphi_{t}\varphi   \,   dx  dy dt +\rho_2 \left[  f(E(t)\int_{\Omega} \psi \psi_t\right]_S^T 
 \\ &+& \displaystyle \!\!  \rho_2\left[  f(E(t))\int_{\Omega} \varphi_{t}\varphi \right]_S^T-K \int_S^T f(E(t))\int_{\Omega} \left((\varphi+w_y)_y w+ (\psi+w_x)_x w \right) \,  dx  dy  dt \\ &+& \displaystyle \!\! K \int_S^T f(E(t))\int_{\Omega}  \left((\varphi+w_y)\varphi  + (\psi+w_x)\psi  \right) \, dx  dy  dt  \\ &-&\displaystyle D \left(\frac{1-\mu}{2}\right) \int_{S}^T f(E(t) \int_{\Omega} \left(\psi_{yy} \psi + \varphi_{xx} \varphi \right) \,dx  dy  dt  \\ &+&\displaystyle \!\! \frac{D}{2}  \int_{S}^T f(E(t) \int_{\Omega} \left( |\psi_{x}|^2 +|\varphi_{y}|^2 \right) \, dx  dy dt\\
   	&-&\displaystyle \!\! D \left(\frac{1+\mu}{2}\right) \int_{S}^T f(E(t) \int_{\Omega}  \left( \psi_{xy} \varphi +  \varphi_{xy} \psi\right)  \, dx  dy  dt \\ &+& \displaystyle\!\! \int_{S}^T f(E(t) \int_{\Omega} \left(  \chi_1(\psi_t)\psi +\chi_2(\varphi_t)\varphi \right)\,  dx dy  dt.
 \end{array}
 \end{equation}
 A straightforward computation yields
\begin{equation}
\begin{array}{lcl}\label{eql1}
  &&\displaystyle D \left(\frac{1-\mu}{2}\right)\int_{S}^T f(E(t) \int_{\Omega}|\psi_y+\varphi_x|^2 dx dy dt\\ &&\quad =
 \displaystyle D \left(\frac{1-\mu}{2}\right) \int_{S}^T f(E(t)) \int_{\Omega} \left(|\psi_{y}|^2 + |\varphi_{x}|^2 \right)  dx  dy dt\\
  &&\quad \displaystyle + D(1-\mu) \int_{S}^T f(E(t)) \int_{\Omega} \psi_y \varphi_x dx dy dt.
\end{array}
 \end{equation}
 Using  \eqref{4441}, \eqref{19}, \eqref{eql1} and taking in account that $v_1^2=v_2^2$, we find
\begin{equation*}
\begin{array}{lcl}\label{eq13}
&&\displaystyle\int_{S}^T f(E(t))  E(t) dt\\ &\leq & \displaystyle\int_{S}^T f(E(t)) \int_{\Omega} \left( \rho_1 |w_t|^2 +\rho_2  |\varphi_t|^2  + \rho_2 |\psi_t |^2 \ + D|\psi_x |^2 +D |\varphi_y |^2 \right)dx  dy  dt\\
&+&\displaystyle K \int_{S}^T f(E(t) \int_{\Omega}\left( |\varphi+w_y|^2+ |\psi+w_x|^2\right)  dx  dy  dt\\
&+&\displaystyle \int_{S}^T f(E(t)) \int_{\Omega} \left( 2D\mu \partial_y \partial _x (\psi \varphi) 
 +2D\mu \psi_x
\varphi_y \right)dx dy dt \\  & + & \displaystyle D \left(\frac{1-\mu}{2}\right)\int_{S}^T f(E(t) \int_{\Omega}|\psi_y+\varphi_x|^2 dx dy dt+CE(S)f(E(S)).
 \end{array}
 \end{equation*}

	Finally, remembering  the definition  of $E$, we end up with the desired  estimate \eqref{domenrg}.
\end{proof}	
	\begin{proposition}\label{lemmado1}
We assume   that the hypotheses of Theorem \ref{thm1}  hold. Then, we have  the following  inequality:
\begin{equation*}\label{ENE}
\int_S^T E(t)f(E(t))dt \leq \sigma E(S),
\end{equation*}
where $\sigma $ is a positive constant given by \eqref{M}.
\end{proposition}

\begin{proof}
 We start by  choosing  $\varepsilon_1$ sufficiently small, namely
$\varepsilon_1=g(\frac{r_0}{\sqrt{2}}).$
Then, for a fixed $t\geq 0$, we define the subset $\Omega_0^t$ as follows:
$$\Omega_0^t=\lbrace (x,y)\in \Omega \;  \; \mbox{such that}\; |\psi_t(t,x,y)| \leq \varepsilon_1 \ \mbox{and}\ |\varphi_t(t,x,y)|\leq \varepsilon_1 \rbrace.$$
Employing \eqref{c3c4} and \eqref{c5c6}, we have
\begin{equation*}
 |\chi_1(\psi_t(t,x,y))|^2\leq ( c_4 g^{-1}(|\psi_t(t,x,y)|)^2 ,  
 \end{equation*}
and
 \begin{equation*}
|\chi_2(\varphi_t(t,x,y))|^2\leq (c_5 g^{-1}(|\varphi_t(t,x,y)|)^2, \quad \forall \ (x,y)\in \Omega_0^t.
\end{equation*}%
 Since $g^{-1}$ is a monotone  increasing function on $\mathbb{R}$, we have
 \begin{equation*}
|\chi_1(\psi_t(t,x,y))|^2 \leq ( c_4 g^{-1}(\epsilon_1))^2 =\frac{c_4^2 r_0^2}{2}, \quad \forall \ (x,y)\in \Omega_0^t,
 \end{equation*}
 and
 \begin{equation*}
 |\chi_2(\varphi_t(t,x,y))|^2 \leq ( c_5 g^{-1}(\epsilon_1))^2 =\frac{c_5^2 r_0^2}{2}, \quad \forall \ (x,y)\in \Omega_0^t.
 \end{equation*}
Hence, we deduce that
 \begin{equation} \label{jensen1}
|\Omega_0^t|^{-1} \int_{\Omega_0^t} |c_4^{-1} \chi_1(\psi_t(t,x,y))|^2 dx dy \in [0,\frac{r_0^2}{2}],
\end{equation}
and
\begin{equation}\label{jensen2}
 |\Omega_0^t|^{-1} \int_{\Omega_0^t} |c_5^{-1}  \chi_2(\varphi_t(t,x,y))|^2 dx dy \in [0,\frac{r_0^2}{2}].
 \end{equation}
On the other hand, $H$ is a strictly convex function on $[0,r_0^2]$, then applying the Jensen's inequality to the quantities in \eqref{jensen1} and \eqref{jensen2}, and using the definition of the function $H(x)=\sqrt{x}g(\sqrt{x})$, we obtain the following estimates:
\begin{equation}
\begin{array}{lcl}\label{p7}
 &&\displaystyle H\left( |\Omega_0^t|^{-1} \int_{\Omega_0^t} |c_4^{-1} \chi_1(\psi_t(t,x,y))|^2  dx dy \right)\\ && \leq
 \displaystyle   |\Omega_0^t|^{-1}  \int_{\Omega_0^t}\frac{1}{c_4} \chi_1(\psi_t(t,x,y)g\left(\frac{1}{c_4} \chi_1(\psi_t(t,x,y))\right) dx dy,
 \end{array}
 \end{equation}
 and
\begin{equation}
\begin{array}{lcl}\label{p8}
&& H\left( |\Omega_0^t|^{-1} \int_{\Omega_0^t} |c_5^{-1}  \chi_2(\varphi_t(t,x,y)|^2  dx dy\right)\\ &&\displaystyle\leq |\Omega_0^t|^{-1}  \int_{\Omega_0^t}\frac{1}{c_5} \chi_2(\varphi_t(t,x,y)g\left(\frac{1}{c_5} \chi_2(\varphi_t(t,x,y))\right) dx dy.
 \end{array}
 \end{equation}
Now, using the fact that $g$ is an increasing function, we have
\begin{equation}\label{eqqq1}
 g\left(\frac{1}{c_4} \chi_1(|\psi_t(t,x,y)|)\right)\leq \psi_t(t,x,y),\hspace{0.2cm} \ \mbox{on}\ \Omega_0^t,
 \end{equation}
 and
 \begin{equation}\label{eqqq2}
g\left(\frac{1}{c_5} \chi_2(|\varphi_t(t,x,y)|\right)\leq
\varphi_t(t,x,y), \hspace{0.2cm} \ \mbox{on}\ \Omega_0^t.
\end{equation}
Knowing that $H(x^2)=xg(x)$ and using the inequalities \eqref{p7}--\eqref{eqqq2}, we infer that
\begin{equation}\label{jj2}
   H\left( |\Omega_0^t|^{-1} \int_{\Omega_0^t} |c_4^{-1} \chi_1(\psi_t(t,x,y))|^2 dx dy \right)\leq |\Omega_0^t|^{-1}\int_{\Omega_0^t} c_4^{-1}\psi_t\chi_1(\psi_t)\; dx  dy,
\end{equation}
and
  \begin{equation}\label{jj1}
 H\left( |\Omega_0^t|^{-1} \int_{\Omega_0^t} |c_5^{-1}  \chi_2(\varphi_t(t,x,y)|^2 dx dy\right)\leq |\Omega_0^t|^{-1}\int_{\Omega_0^t} c_5^{-1}\varphi_t\chi_2(\varphi_t) dx  dy.
 \end{equation}
 Since $\chi_{i}$,  for $i=1,2$, are  increasing  functions, then we obtain
\begin{eqnarray*}
|\Omega^t_0|^{-1} c_4^{-1}\int_{\Omega^t_0} \chi_1(\varphi_t(t,x,y))\varphi_t(t,x,y) dx  dy
&\leq& |\Omega^t_0|^{-1} c_4^{-1}\int_{\Omega^t_0}\chi_1(\varepsilon_1)\varepsilon_1   dx  dy \nonumber \\ &\leq& \varepsilon_1 g^{-1}(\varepsilon_1)= H(\frac{r_0^2}{2}),
\end{eqnarray*}
and
\begin{eqnarray*}
|\Omega^t_0|^{-1} c_5^{-1}\int_{\Omega^t_0} |\chi_2(\psi_t(t,x,y))||\psi_t(t,x,y)|  dx  dy &\leq& |\Omega^t_0|^{-1} c_5^{-1}\int_{\Omega^t_0}\chi_2(\varepsilon_1)\varepsilon_1  dx   dy \nonumber \\ &\leq& \varepsilon_1 g^{-1}(\varepsilon_1)= H(\frac{r_0^2}{2}).
\end{eqnarray*}
Then, we have
\begin{equation*}\label{eqqt1}
|\Omega^t_0|^{-1} c_4^{-1}\int_{\Omega^t_0} |\chi_1(\psi_t(t,x,y))||\psi_t(t,x,y)| dx  dy\in \left[0,H\left(\frac{r_0^2}{2}\right)\right],
\end{equation*}
and
\begin{equation*}\label{eqqt2}
|\Omega^t_0|^{-1} c_5^{-1}\int_{\Omega^t_0} \chi_2(\varphi_t(t,x,y))\varphi_t(t,x,y) dx dy \in\left[0,H\left(\frac{r_0^2}{2}\right)\right].
\end{equation*}
Consequently,  we get
 \begin{equation*}\label{assupH2}
 H^{-1}\left( |\Omega_0^t|^{-1}  \int_{\Omega_0^t} c_4^{-1}  \chi_1(\psi_t(t,x,y))\psi_t(t,x,y) dx dy \right) \in [0,\frac{r_0^2}{2}],
 \end{equation*}
and
\begin{equation*}\label{assupH}
 H^{-1}\left( |\Omega_0^t|^{-1}  \int_{\Omega_0^t} c_5^{-1}  \chi_2(\varphi_t(t,x,y))\varphi_t(t,x,y) dx dy \right) \in [0,\frac{r_0^2}{2}].
 \end{equation*}
Therefore, from the  inequalities \eqref{jj2} and \eqref{jj1}, we deduce that
\begin{equation*}
\begin{array}{lcl}
&&\displaystyle\int_S^T f(E(t)) \left(\int_{\Omega_0^t} |\chi_1(\psi_t)|^2 +\int_{\Omega_0^t} |\chi_2(\varphi_t)|^2\right) \\
&& \displaystyle\leq c_4^2 \int_{S}^T f(E(t)) H^{-1}\left( |\Omega^t_0| c_4^{-1}\int_{\Omega^t_0} \chi_1(\psi_t(t,x,y))\psi_t(t,x,y)\right) \\
&&+ \displaystyle c_5^2 \int_{S}^T f(E(t))  H^{-1}\left(
|\Omega^t_0| c_5^{-1}\int_{\Omega^t_0} \chi_2(\varphi_t(t,x,y))\varphi_t(t,x,y)\right). 
\end{array}
\end{equation*}
Indeed, applying the following Young's inequality
\begin{equation}\label{young}
AB\leq \widehat{H}^{\star}(A)+ \widehat{H}(B),
\end{equation}
 for $A= f(E(t))$ and  for the  two values of $B$ as below; first we consider $$B=B_1(t):= H^{-1}\left( |\Omega^t_0|^{-1} c_4^{-1}\int_{\Omega^t_0} \chi_1(\psi_t(t,x,y))\psi_t(t,x,y)\right),$$
 and second $$B= B_2(t):= H^{-1}\left( |\Omega^t_0|^{-1} c_5^{-1}\int_{\Omega^t_0} \chi_2(\varphi_t(t,x,y))\varphi_t(t,x,y)\right).$$
Hence, we have
\begin{equation}
\begin{array}{lcl}\label{omega1}
&&\displaystyle\int_S^T f(E(t)) \left(\int_{\Omega_0^t} |\chi_1(\psi_t)|^2dx dy +\int_{\Omega_0^t} |\chi_2(\varphi_t)|^2dx dy\right) dt\\
\quad&\displaystyle\leq& \displaystyle c_4^2  |\Omega_0^t| \int_S^T \widehat{H}^{\star}(f(E(t))) dt
 + |\Omega_0^t| \left( |\Omega^t_0|^{-1} c_4\int_{\Omega^t_0} \chi_1(\psi_t(t,x,y))\psi_t(t,x,y)dxdy\right)\\
&+& \displaystyle c_5^2 |\Omega_0^t| \int_S^T \widehat{H}^{\star}(f(E(t))) dt + |\Omega_0^t| \left(
|\Omega^t_0|^{-1} c_5\int_{\Omega^t_0} \chi_2(\varphi_t(t,x,y))\varphi_t(t,x,y)dxdy\right). 
\end{array}
\end{equation}
Taking into account the  dissipation relation  \eqref{dissip}, it is clear that, for all $0\leq S \leq T$,
\begin{equation}\begin{array}{lcl}\label{jj3}
\displaystyle\int_{S}^T \int_{\Omega_0^t} \left(c_4 \psi_t \chi_1(\psi_t) + c_5\varphi_t \chi_2(\varphi_t)\right)dxdydt&\leq & \displaystyle \max(c_4,c_5)\; (E(S)-E(T))\\
&\leq & \max(c_4,c_5)\; E(S).
\end{array}
\end{equation}
Using  \eqref{omega1} and \eqref{jj3}, we deduce that
\begin{equation*}\begin{array}{lcl}
&&\displaystyle\int_S^T f(E(t)) \left(\int_{\Omega_0^t} |\chi_1(\psi_t)|^2dxdy +\int_{\Omega_0^t} |\chi_2(\varphi_t)|^2 dxdy \right) dt \\&& \displaystyle\leq \left( (c_4^2+c_5^2) |\Omega_0^t| \int_S^T  \widehat{H}^{\star}(f(E(t))dt\right)  + \int_S^T \int_{\Omega_0^t} \left(c_4 \psi_t \chi_1(\psi_t)+c_5 \varphi_t\chi_2(\varphi_t)\right)dxdy \, ds \\
& &\displaystyle \leq (c_4^2+c_5^2) |\Omega_0^t| \int_S^T  \widehat{H}^{\star}(f(E(t))dt +
\max(c_4,c_5)  E(S),\
\forall  \  0 \leq S \leq  T.
\end{array}
\end{equation*}
Therefore, for $ 0   \leq S \leq  T$, we have
\begin{equation}\begin{array}{lcl}\label{omega}
&&\displaystyle\int_S^T f(E(t)) \int_{\Omega \setminus \Omega_0^t} \left(|\chi_1(\psi_t)|^2 + |\chi_2(\varphi_t)|^2\right) dxdy\, dt \\ &&\leq \displaystyle  \int_S^T
f(E(t))  \int_{\Omega \setminus \Omega_0^t}  \left(c_4 |\psi_t \chi_1(\psi_t)| + c_5 |\varphi_t\chi_2(\varphi_t)| \right) dxdy\, dt\\
&&\leq \displaystyle \max(c_5,c_4) \; \int_S^T f(E(t)) (-E'(t)) dt\\
 &&\leq \displaystyle \max(c_5,c_4) E(S) f(E(S)).  
\end{array}
\end{equation}
Finally, using the inequalities \eqref{omega1} and \eqref{omega}, we obtain for $0 \leq S \leq  T,$
\begin{equation*}\begin{array}{lcl}
\displaystyle\int_S^T  f(E(t)) \int_{\Omega}( |\chi_1(\psi_t)|^2\!&+&\!|\chi_2(\varphi_t)|^2) dx dy dt \leq  \displaystyle(c_4^2+c_5^2) |\Omega_0^t| \int_S^T  \widehat{H}^{\star}(f(E(t))) dt \\
\!&+&\! \displaystyle\max(c_4,c_5) \; E(S)(f(E(S)))+ \max(c_4,c_5) \; E(S).
\end{array}
\end{equation*}

Now, the second step of the proof starts by choosing   a sufficiently small number 
$$\varepsilon_2=\min(r_0,g(r_1)),$$
where $r_1$ is a positive  constant such that $H(r_1^2)=\max(\frac{c_3}{c_4},\frac{c_5}{c_6}) H(r_0^2)$.

Thanks to $(H_0)$ there exist constants, still denoted here by  $c_3>0$  and $c_4>0$  to avoid too many notations, such that
\begin{equation}\label{h21}
\hspace{0.3cm}\left\{
\begin{array}{l}
c_3|v|\leq |\chi_1(v)|\leq c_4|v|, \hspace{2.5cm} \forall \ |v|\geq \varepsilon_2,\\
c_3|g(|v|)|\leq |\chi_1(v)|\leq c_4 g^{-1}(|v|), \hspace{1cm} \forall \ |v|\leq \varepsilon_2.
\end{array}%
\right.
\end{equation}%
Moreover, there exist constants, still denoted also by  $c_5>0$  and $c_6>0$,  such that
\begin{equation}\label{h211}
\hspace{0.3cm}\left\{
\begin{array}{l}
c_6|v|\leq |\chi_2(v)|\leq c_5|v|, \hspace{2.5cm} \forall \ |v|\geq \varepsilon_2,\\
c_6|g(|v|)|\leq |\chi_2(v)|\leq c_5 g^{-1}(|v|), \hspace{1cm} \forall \ |v|\leq \varepsilon_2.
\end{array}%
\right.
\end{equation}%
At this level, we need to  introduce two sets  $G_0^t $  and $G_1^t$ which are defined as follows:
$$G_0^t=\lbrace (x,y)\in \Omega , |\psi_t(t,x,y)| \leq  \epsilon_2 \rbrace, \; \; \forall \; t\geq 0,$$
and
$$G_1^t=\lbrace (x,y)\in \Omega , |\varphi_t(t,x,y)| \leq  \epsilon_2 \rbrace, \; \; \; \forall \; t\geq 0.$$
Then,  for all $t\geq 0$, we have
\begin{equation*}
|G_0^t|^{-1}\int_{G_0^t} |\psi_t(t,x,y)|^2  dx \ dy \in [0,r_0^2],
\end{equation*}
and
\begin{equation*}
|G_1^t|^{-1}\int_{G_1^t} |\varphi_t(t,x,y)|^2  dx \ dy  \in [0,r_0^2].
\end{equation*}
Using the inequalities   \eqref{h21} and \eqref{h211}, we obtain
$$g(|\psi_t|)\leq c_3^{-1}\chi_1(\psi_t)\ \mbox{on}\ G_1^t, \quad \text{and} \quad g(|\varphi_t|)\leq c_6 ^{-1}\chi_2(\varphi_t)\  \mbox{on}\  G_1^t.$$
Following the same  arguments as in \eqref{p7},  we obtain
\begin{equation*}\begin{array}{lcl}\label{pp7}
\displaystyle H\left( |G_0^t|^{-1} \int_{G_0^t} |\psi_t|^2  dx dy  \right) &\leq & \displaystyle |G_0^t|^{-1}\int_{G_0^t} H(|\psi_t(t,x,y)|^2)\;  dx \;  dy  \\
&\le&  \displaystyle |G_0^t|^{-1} \int_{G_0^t} |\psi_t | g(|\psi_t(t,x,y)|)\; dx \; dy \\ &\leq & \displaystyle \frac{1}{c_6}\int_{G_0^t}\int_{G_0^t}|\varphi_t | |\chi_2(\varphi_t)|  \; dx \; dy,
\end{array}
\end{equation*}
and
\begin{equation*}\begin{array}{lcl}\label{pp8}
\displaystyle H\left( |G_1^t|^{-1} \int_{G_0^t} |\varphi_t|^2 dx dy  \right) &\leq &  |G_1^t|^{-1}\int_{G_1^t} H(|\varphi_t(t,x,y)|^2) dx  dy  \\
&\le & \displaystyle |G_1^t|^{-1} \int_{G_1^t} |\varphi_t | g(|\varphi_t(t,x,y)|) dx  dy\\
&\leq &\displaystyle \frac{1}{c_3}|G_1^t|\int_{G_1^t}|\psi_t||\chi_1(\psi_t)| dx  dy.
\end{array}
\end{equation*}

Hence, since $H$ is an increasing function, we have
$$\begin{array}{lcl}
&&\displaystyle\int_S^T f(E(t))\int_{G_1^t} |\varphi_t |^2 dx dy dt\\ &&\displaystyle \leq \int_S^T f(E(t))
 |G_1^t| H^{-1}\left(  |G_1^t|^{-1}  c_6^{-1}\int_{G_1^t} |\varphi_t| | \chi_2(\varphi_t)| dx dy\right) dt,\end{array}$$
 and
 $$\begin{array}{lcl}
&&\displaystyle\int_S^T f(E(t))\int_{G_0^t} |\psi_t |^2 dx dy dt\\ &&\displaystyle\leq \int_S^T f(E(t))
 |G_0^t| H^{-1}\left(  |G_0^t|^{-1}  c_4^{-1}\int_{G_1^t} |\psi_t| | \chi_1(\psi_t)|dx dy\right) dt.\end{array}$$
 Now, we use the Young's inequality \eqref{young} with $A= f(E(t))$ and $B$ taking successively the two following values:
  $$B=B_3(t):=H^{-1}\left(  |G_0^t|^{-1}  c_3^{-1}\int_{G_1^t} |\psi_t| | \chi_1(\psi_t)|dx dy\right),$$
and $$B=B_4(t):=H^{-1}\left(  |G_1^t|^{-1}  c_6^{-1}\int_{G_1^t} |\varphi_t| | \chi_2(\varphi_t)|dx dy\right).$$
Therefore, we deduce that
 \begin{equation*}\begin{array}{lcl}\label{k11}
&&\displaystyle\int_S^T f(E(t))\int_{G_0^t} |\psi_t |^2 dx dy dt\\ &&\displaystyle \leq  |G_0^t|\int_S^T \bigg[\widehat{H}^{\star}(f(E(t)))+\widehat{H}\left(H^{-1}( |G_0^t|^{-1} \frac{1}{c_3} \int_{G_0^t} \chi_1(\psi_t)\psi_t )\right)\bigg] dt,
 \end{array}\end{equation*}
and
 \begin{equation*}\begin{array}{lcl}\label{k22}
&&\displaystyle \int_S^T f(E(t))\int_{G_1^t} |\varphi_t |^2dx dy dt 
\\&& \leq \displaystyle  |G_1^t|\int_S^T \bigg[ \widehat{H}^{\star}(f(E(t)))+\widehat{H}\left(H^{-1}(|G_1^t|^{-1} c_6^{-1}\int_{G_1^t}\varphi_t \chi_2(\varphi_t))\right)\bigg] dt.
 \end{array} \end{equation*}
 On the other hand, thanks to \eqref{h21} and \eqref{h211}, we have
\begin{equation*}\begin{array}{lcl}
\displaystyle |G_0^t|^{-1}  c_3^{-1}\int_{G_1^t} \psi_t \chi_1(\psi_t)dx dy &&\leq \displaystyle c_4 |G_0^t|^{-1}c_3^{-1} \int_{G_0^t} \psi_t g^{-1}(\psi_t)dx dy  \\ &&\displaystyle\leq \frac{c_4}{c_3} \varepsilon_1 g^{-1}(\varepsilon_1)=\frac{c_4}{c_3}H(r_1^2)\\ &&\displaystyle\leq \max(\frac{c_4}{c_3}, \frac{c_5}{c_6}) H(r_1^2)=H(r_0^2).
 \end{array}\end{equation*}
Similarly, we obtain
\begin{equation*}\begin{array}{lcl}
\displaystyle   |G_1^t|^{-1}  c_6^{-1}\int_{G_1^t} |\varphi_t| | \chi_2(\varphi_t)|dx dy &\leq & \displaystyle c_5 |G_0^t|^{-1}c_6^{-1} \int_{G_0^t} \varphi_t g^{-1}(\varphi_t) dx dy \\ &\leq &\displaystyle \frac{c_5}{c_6} \varepsilon_1 g^{-1}(\varepsilon_1)
 =\frac{c_5}{c_6}H(r_1^2) \\ & \leq & \displaystyle \max(\frac{c_4}{c_3}, \frac{c_5}{c_6} ) H(r_1^2) = H(r_0^2).
\end{array}\end{equation*}
   Using the fact that $H$ is an increasing function, we have
$$
  H^{-1}\left( |G_0^t|^{-1} \frac{1}{c_3} \int_{G_0^t} \chi_1(\psi_t)\psi_t\right) \in [0,r_0^2],$$
and
$$H^{-1}\left(|G_1^t|^{-1} c_6^{-1}\int_{G_1^t}\varphi_t \chi_2(\varphi_t)\right)\in [0,r_0^2].
$$

Remember that   $\widehat{H}(x)=H(x)$, for all $x\in [0,r_0^2]$, and using the dissipation formula \eqref{dissip},  we find that
  \begin{equation*}
 \int_S^T f(E(t))\int_{G_0^t} |\psi_t |^2 dx dy dt \leq   |G_0^t|\int_S^T \widehat{H}^{\star}(f(E(t)))dt+\frac{1}{c_3} E(S), \; \forall \; 0\leq S\leq T,
   \end{equation*}
   and
   \begin{equation*}
\int_S^T f(E(t))\int_{G_1^t} |\varphi_t |^2 dx dy dt  \leq  |G_1^t|\int_S^T \widehat{H}^{\star}(f(E(t)))dt+\frac{1}{c_6} E(S),\; \forall \; 0\leq S\leq T.
     \end{equation*}
  As for \eqref{omega}, we conclude  that
 $$\int_S^T f(E(t)) \int_{\Omega \setminus G_0^t}|\psi_t|^2 dx dy dt\leq c_3^{-1} E(S) f(E(S)),\ \forall \ 0\leq S \leq T, $$
and
 $$\int_S^T f(E(t)) \int_{\Omega \setminus G_1^t}|\varphi_t|^2 dx dy dt\leq c_6^{-1} E(S) f(E(S)), \ \forall \ 0\leq S \leq T.$$
 Consequently, we have
\begin{equation}\begin{array}{lcl}\label{equality2}
 \displaystyle\int_{S}^T f(E(t)) \int_{\Omega} |\psi_t |^2 +|\varphi_t |^2 dx dy dt  &\leq & \displaystyle   \left( \frac{1}{c_3} +\frac{1}{c_6} \right)  E(S) f(E(S))\\
  &+& \displaystyle  (|G_0^t|+ |G_1^t|) \int_S^T \widehat{H}^{\star}(f(E(t)))dt \\
 &+&
  \displaystyle\left(\frac{1}{c_3} + \frac{1}{c_6}\right) E(S), \quad \forall \ 0\leq S \leq T.
  \end{array}
  \end{equation}
Finally, using  the  estimates  \eqref{domenrg}, \eqref{omega1} and \eqref{equality2}, we obtain
\begin{equation}\begin{array}{lcl}\label{h11}
 \displaystyle \int_S^T E(t)f(E(t))\; dt &\leq & \displaystyle \left( \alpha_3 \left( \frac{1}{c_3} +\frac{1}{c_6} \right) +  \alpha_2 (c_4+c_5) \right)  E(S) f(E(S)) \\
  &+&\displaystyle\left(  \alpha_3 (|G_0^t|+ |G_1^t|) +  \alpha_2  (c_4^2+c_5^2) |\Omega_0^t| \right) \int_S^T \widehat{H}^{\star}(f(E(t)))\\
 &+&\displaystyle\alpha_3 \left(\frac{1}{c_3}+ \frac{1}{c_6} \right) E(S),\hspace{0.5cm}
 \forall \ 0\leq S \leq T.
  \end{array}
  \end{equation}
  In order to achieve the proof of this proposition, we need to estimate the term $f(E(S))$. For that purpose, we  define the constant $\beta$, which depends on $E(0)$, as follows:
   \begin{equation*}\label{beta1}
  \beta=\max \left( c_3, \frac{E(0)}{2L(H'(r_0^2))} \right),
   \end{equation*}
  where  $L$ is defined in \eqref{L} and  $c_3=\alpha_3 \left(  |G_0^t|+ |G_1^t|\right) +  \alpha_2 (c_4^2+c_5^2) |\Omega_0^t| $. Here $c_3$ is a positive constant which depends on the physical characteristics of the plate ($\rho_1, \rho_2, b$ and $K$) but it is independent of $E(0)$.
  
Now,  using the definition  of $f$, given by \eqref{ffff1},  the fact that $E$ is a non-decreasing function and the inequality \eqref{bt1}, we conclude that
  $$\frac{E(t)}{2\beta} \leq \frac{E(0)}{2\beta}\leq L(H'(r_0^2))<r_0^2.$$
  Recall that $f$ is an increasing function, we obtain
  $$f(E(S))\leq f(E(0))=L^{-1} \left( \frac{E(0)}{2\beta}\right)\leq H'(r_0^2), \quad \forall \  S\geq 0.$$
  Thus,  the inequality \eqref{h11} reduces to
\begin{equation*}\begin{array}{lcl}
 \displaystyle \int_S^T E(t)f(E(t))\; dt &\leq & \displaystyle \left( \alpha_3 \left( \frac{1}{c_3} +\frac{1}{c_6} \right) +  \alpha_2 (c_4+c_5) \right)  E(S) f(E(S)) \\
  &+& \displaystyle\left(  \alpha_3 (|G_0^t|+ |G_1^t|) +  \alpha_2  (c_4^2+c_5^2) |\Omega_0^t| \right) \int_S^T \widehat{H}^{\star}(f(E(t)))dt\\
 &+& \displaystyle\alpha_3 \left(\frac{1}{c_3}+ \frac{1}{c_6} \right) E(S),\vspace{0.1cm}\\
  &\leq &  \displaystyle\left( \alpha_3\left( \frac{1}{c_3} +\frac{1}{c_6}  \right) (H'(r_0^2)\;  \alpha_2  (c_4+c_5) +1) \right) E(S) \vspace{0.1cm}\\
  &+& \displaystyle C_3 \int_S^T \widehat{H}^{\star}(f(E(t)))dt,  \hspace{0.5cm}
 \forall \ 0\leq S \leq T.
  \end{array}
  \end{equation*}
  Following the same choice for the function $f$ as in \cite{R4}, we have
  $$C_3   \widehat{ H}^{\star}(f(s)) \leq \beta \widehat{ H}^{\star}(f(s))=\frac{sf(s)}{2}, \quad \forall \ s\in [0,\beta r_0^2).$$
  The above estimate holds true for $s=E(S)$, hence, we deduce that
  \begin{equation*}\label{eq123}
  \int_S^T E(t)f(E(t)) dt \leq  \sigma  E(S),  \hspace{0.5cm}
 \forall \ 0\leq S \leq T,
  \end{equation*}
  where
  \begin{equation}\label{sig}
  \sigma = 2 \left(\alpha_3 \left( \frac{1}{c_3} +\frac{1}{c_6}  \right)( H'(r_0^2) \alpha_2 (c_4+c_5)+1)\right).
  \end{equation}
  
  This completes the proof of Proposition \ref{lemmado1}.
\end{proof}

In the sequel, we will put together all the necessary estimates and then prove the results in Theorem \ref{th2} which will be the subject of the next subsection.
\subsection{Proof of Theorem \ref{th2}}
In this subsection, we conclude the proof of Theorem \ref{th2} for which we already prepared all the necessary material. 
   \begin{proof} (of Theorem \ref{th2})\mbox{}\\
    Thanks to the dissipative relation \eqref{dissip}, the energy $E(t)$, defined by \eqref{en} and associated with the solution  of the system \eqref{MTD1}-\eqref{MTD2}, is a decreasing and absolutely continuous function from $[0,\infty)$ on $[0,\infty)$. Moreover, this energy satisfies
     \begin{equation*}
  \int_S^T E(t)f(E(t))dt \leq  \sigma E(S),  \hspace{0.5cm}
 \forall \ 0\leq S \leq T,
  \end{equation*}
  where $\sigma $ is a constant given by \eqref{sig}, and $$f(E(t))=L^{-1}\left(\frac{E(t)}{2\beta}\right).$$
  Hence, by applying the results in \cite[Theorem 2.3]{R4}, we deduce that $E(t)$ satisfies the estimate \eqref{energy-est}. This concludes the proof of Theorem \ref{th2}.
 \end{proof}


\section{Examples}\label{examples}
 Through two illustrative examples, we apply the result of Theorem \ref{th2}, more precisely the inequality \eqref{dcry}, in order to characterize the  behavior of the energy as time goes to infinity and establish an explicit decay rate for the energy. 
\subsection{Example 1}
We focus on the case where $\chi_1 = \chi_2$ and satisfy the assumption $(H_0)$,  and $g(x) = cx^p$. Specifically, we assume that for some constants $c_1$, $c_2$, and $p > 1$, the following inequalities hold:
\begin{equation}
    c_1 \min(|x|,|x|^p) \leq |\chi_i(x)| \leq c_2 \max(|x|,|x|^{\frac{1}{p}}).
\end{equation}
The function $H(x)$ takes the form $H(x) = cx^{\frac{p+1}{2}}$,  leading to 
$ \limsup_{x \to 0^{+}}\Lambda_{H}(x)= \frac{2}{p+1} <1.$ Consequently, formula \eqref{dcry} provides us with the following explicit decay rate for the energy estimate \eqref{en},
\begin{equation}
    E(t)\leq \beta \left(\frac{\sigma}{t}\right)^{\frac{2}{p+1}},
\end{equation}
where the constants $\sigma$ and $\beta$ are  given, respectively, by \eqref{beta} and \eqref{M}.
\subsection{Example 2}
For $g(x)= \exp \left({\frac{-1}{x^2}}\right)$, we have $H(x)= \sqrt{x}\exp\left(\frac{-1}{x^2}\right)$. This yields $$\Lambda_{H}(x)= \frac{2x}{x+2}.$$
For this case $\limsup_{x \to 0^+} \Lambda_{H}(x)=0 <1$. Hence,  formula \eqref{energy-est} provides us with the following explicit decay rate for the energy
\begin{equation}
E(t)\leq \beta \left(\ln\left(\frac{\sigma}{t}\right)\right)^{-1},
\end{equation}
where the constants $\sigma$ and $\beta$ are  given, respectively, by \eqref{beta} and \eqref{M}.

\end{document}